\documentclass[11pt,twoside,a4paper,leqno]{article}
\usepackage{a4}
\usepackage[utf8x]{inputenc}
\usepackage[english]{babel}
\usepackage[T1]{fontenc}
\usepackage[usenames,dvipsnames]{xcolor}
\usepackage{amsmath,amsthm}
\usepackage{amsfonts,amssymb}
\usepackage{array}

\usepackage{float}

\usepackage{color} 
      \definecolor{bleu_sombre}{rgb}{0,0,0.6}  
     \definecolor{rouge_sombre}{rgb}{0.8,0,0}
    \definecolor{vert_sombre}{rgb}{0.1,0.4,0.1}
\usepackage{etex}
\usepackage{tikz}
\usepackage{pgfkeys}
\usepackage{pgffor}
\usepackage{pgfcalendar}
\usetikzlibrary{arrows,calc,matrix,topaths,positioning,scopes,shapes,decorations}
 \usepackage{mathpazo}
\usepackage{microtype} 
\usepackage[numbers]{natbib}
\usepackage{fancyhdr}			%en-tete
\DeclareMathOperator{\Ker}{Ker}

\DeclareMathOperator{\Hom}{Hom}
\DeclareMathOperator{\ds}{s^{-1}}
\DeclareMathOperator{\s}{s}

\DeclareMathOperator{\Ba}{\mathcal{B}}

\newtheorem*{rep@theo}{\rep@title}
\newcommand{\newreptheo}[2]{%
\newenvironment{rep#1}[1]{%
 \def\rep@title{#2 \ref{##1}}%
 \begin{rep@theo}}%
 {\end{rep@theo}}}

\newtheorem*{rep@prop}{\rep@title}
\newcommand{\newrepprop}[2]{%
\newenvironment{rep#1}[1]{%
 \def\rep@title{#2 \ref{##1}}%
 \begin{rep@prop}}%
 {\end{rep@prop}}}

\newtheorem*{rep@cor}{\rep@title}
\newcommand{\newrepcor}[2]{%
\newenvironment{rep#1}[1]{%
 \def\rep@title{#2 \ref{##1}}%
 \begin{rep@cor}}%
 {\end{rep@cor}}}
\makeatother

\title{On the homology of the double cobar construction of a double suspension}
\author{Alexandre Quesney \thanks{%\vspace{2cm}
{Laboratoire de Math\'ematiques Jean-Leray, 2 rue de la Houssini\`ere - BP 92208 F-44322 Nantes Cedex 3, France.} 
E-mail address: alexandre.quesney@univ-nantes.fr
}
}
\date{}

\begin{document}

\maketitle

\theoremstyle{definition}
\newtheorem{de}{Definition}[section]
\theoremstyle{remark}
\newtheorem{rmq}[de]{Remark}
\newtheorem{exple}[de]{Example}
\theoremstyle{plain}
\newtheorem{theo}[de]{Theorem}
\newtheorem{prop}[de]{Proposition}
\newtheorem{lem}[de]{Lemma}
\newtheorem{cor}[de]{Corollary}
\newrepcor{cor}{Corollary}
\newrepprop{prop}{Proposition}
\newtheorem{properties}[de]{Properties}
\newreptheo{theo}{Theorem}
\numberwithin{equation}{section}

\newtheorem{proj}{Projet}

\newtheorem*{theo*}{Theorem}
\newtheorem*{prop*}{Proposition}
\newtheorem*{cor*}{Corollary}

\newcommand{\cob}{\Omega C_* (X)}
\newcommand{\ccob}{\Omega^2 C_* (X)}
\newcommand{\Hopf}{\mathcal{H}}
\newcommand{\ot}{\otimes} 
\newcommand{\diskf}{\mathcal{C}^f_2} 
\newcommand{\disk}{\mathcal{C}_2}
\newcommand{\cobH}{\Omega \mathcal{H}}
\newcommand{\exd}{\widetilde{\sigma}} 
\newcommand{\tra}{\mathcal{T}}
\newcommand{\Zdeux}{\mathbb{Z}/2\mathbb{Z}}
\newcommand{\os}{\Omega \Sigma}
\newcommand{\La}{A}
\newcommand{\ov}[1]{\overline{#1}}
\newcommand{\OS}[1]{\Omega^{#1} \Sigma^{#1}}
\newcommand{\uv}[1]{\underline{#1}}
\newcommand{\wid}[1]{\widetilde{#1}}
\newcommand{\att}[1]{\textcolor{red}{#1}}
\newcommand{\note}[1]{\marginpar{\begin{footnotesize}\att{#1}\end{footnotesize}}}
\newcommand{\ca}[1]{\mathcal{#1}}

\setcounter{tocdepth}{1}
%\tableofcontents

\begin{abstract}
The double cobar construction of a double suspension comes with a Connes-Moscovici structure, that is a homotopy G-algebra (or Gerstenhaber-Voronov algebra) structure together with a particular BV-operator up to a homotopy. 
We show that  
the homology of the double cobar construction of a double suspension is a free BV-algebra. In characteristic two, a similar result holds for the underlying $2$-restricted Gerstenhaber algebra. These facts rely on a formality theorem for the double cobar construction of a double suspension.
\end{abstract}

\section*{Introduction}

For $n\geq 1$, the $n$-fold loop spaces (of $SO(n)$-spaces) are characterized by the (framed) little $n$-discs operad. 
In case $n=1$, the cobar construction provides an algebraic model for the chain complex of loop spaces; it can be iterated to form the double cobar construction that provides a model for the chain complex of $2$-fold loop spaces. 

We are interested in further algebraic structures on the double cobar construction encoded by operads that are either equivalent to the chain little $2$-discs operad, or equivalent to the chain \emph{framed} little $2$-discs operad in the case of $2$-fold loop spaces of $S^1$-spaces. 
Recall that the homology operad of the little $2$-discs operad is the Gerstenhaber operad and that the homology operad of the framed little $2$-discs operad is the Batalin-Vilkovisky operad. 

Let $\Omega^2 Y$ be the double  loop space of a pointed topological space $Y$. 
Then its rational homology $H_*(\Omega^2 Y)$ is a Gerstenhaber algebra:
\begin{itemize}
 \item the graded commutative product is the Pontryagin product $-\cdot-$; 
\item the degree $1$ Lie bracket is the Browder bracket $\{-;-\}$; 
\item they are compatible in the following way: 
\begin{align}\label{eq: Poisson}
\{a;b\cdot c\} = \{ a; b \}\cdot c +(-1 )^{(|a|+ 1)|b|} b\cdot\{ a;c\},
\end{align}
for all homogeneous $a,b,c \in H_*(\Omega^2 Y)$.
\end{itemize}
If $Y$ is endowed with an action of the circle $S^1$ (preserving the base point) then the double loop space $\Omega^2 Y= \Hom(S^2,Y)$ is also an $S^1$-space with the following action:
\begin{align*}
\Delta: S^1\times \Hom(S^2,Y)&\to \Hom(S^2,Y)\\
 (g,\phi)&\mapsto (g\cdot \phi)(s):=g\cdot \phi (g^{-1}\cdot s),%
\end{align*}
where the circle $S^1$ acts on the 2-sphere $S^2$ by the base point-preserving rotations.
This action induces (together with the Pontryagin product) a Batalin-Vilkovisky algebra structure $(H_*(\Omega^2 Y),\Delta)$, \cite{Getzler}.

In \cite{Quesney} the author shows  that the double cobar construction of a double suspension is endowed with a Connes-Moscovici structure.
Essentially it is a homotopy G-algebra structure together with the boundary operator defined by A. Connes and H. Moscovici in \cite{ConnesMosc}, that is a BV-operator up to a homotopy. This can be seen as a first step for obtaining a strong homotopy BV-algebra structure on the double cobar construction. 

Our main structural results are the followings.%
\begin{theo*}[\ref{th: BV sur 2Cobar}]
 Let the coefficient field be $\mathbb{Q}$.
 Then $\Omega^2 C_*(\Sigma^2X)$ has a Connes-Moscovici structure. The induced BV-algebra structure $(H_*(\Omega^2 C_*(\Sigma^2 X)),\Delta_{CM})$ is free on the reduced homology $H^+_*(X)$.
\end{theo*}
Therefore 
the Batalin-Vilkovisky algebras $(H_*(\Omega^2 C_*(\Sigma^2X)),\Delta_{CM})$ and\\
$(H_*(\Omega^2|\Sigma^2 X|),\Delta)$ are isomorphic.

In characteristic 2 the homology of double loop spaces is 2-restricted Gerstenhaber algebra (see Definition \ref{de: Gerst restreinte}).%
We obtain %
\begin{theo*}[\ref{th: GerstLibreCobar1restreint}]
Let the coefficient field be $\mathbb{F}_2$. 
Then $\Omega^2 C_*(\Sigma^2X)$ has a Connes-Moscovici structure. 
The induced  $2$-restricted Gerstenhaber algebra structure on $H_*(\Omega^2 C_*(\Sigma^2 X))$ is free on the reduced homology  $H^+_*(X)$.
\end{theo*}
Therefore the restricted Gerstenhaber algebras $H_*(\Omega^2 C_*(\Sigma^2X))$ and  \\
 $H_*(\Omega^2|\Sigma^2 X|)$ are isomorphic.

The two previous theorems use the following formality theorem. 

 Let us mention that what we call a homotopy BV-algebra cf.\cite{Quesney} has to be understood as a homotopy G-algebra with a BV-operator up to a homotopy. 
 Such a structure has the Connes-Moscovici structure as an archetype. 
\begin{theo*}[\ref{th: formality}]
 Let $\Bbbk$ be any coefficient field. 
Let $C$ be a homotopy G-coalgebra with no higher structural co-operations and such that any element of $C^+$ is  primitive for the coproduct.
Then $\Omega^2 C$ and $\Omega^2 H_*(C)$, endowed with the Connes-Moscovici structure, are quasi-isomorphic as homotopy BV-algebras.
 \end{theo*}
The main example of such a $C$ is the simplicial chain complex of a double suspension.
\newpage

The paper is organized as follows.

In the first section: we fix notations, we recall both the cobar construction and what are homotopy G-(co)algebras.
The main result is the Proposition \ref{prop: reduced G-cog}, coming from \cite{Quesney}, that claims the existence of the Connes-Moscovici structure on the double cobar construction of a double suspension.
The second section is devoted to the formality theorem.
The third section described the $2$-restricted Gerstenhaber algebras coming from the homology of a homotopy G-algebra.
In the fourth section we prove that, in characteristic two, the homology of the double cobar construction of a double suspension is a $2$-restricted Gerstenhaber algebra.
In the last section the coefficient field is $\mathbb{Q}$; 
we prove that the homology of the double cobar construction of a double suspension is BV-algebra.
In the appendix we write the relations among the structural co-operations of homotopy G-coalgebras.

\section{A few algebraic structures on the cobar construction}
This section does not contain new results. 
Here we recall a few algebraic structures on both the cobar and the double cobar constructions.
The main result of this section shows that 
the double cobar construction of an \emph{involutive} homotopy G-coalgebra has a \emph{Connes-Moscovici structure}. 
Roughly speaking, a Connes-Moscovici structure is a \emph{homotopy G-algebra} structure together with a BV-operator up to a  homotopy. The terminology Connes-Moscovici structure  refers to the boundary operator defined on the cobar construction of an involutive Hopf algebra in \cite{ConnesMosc}; it is here the BV-operator.
\\

The cobar construction is a functor
\begin{align*} 
 \Omega : DGC &\to DGA% 
 \end{align*}
from the category of $1$-connected differential graded coalgebras to the category of connected differential graded algebras. By a $1$-connected dg-coalgebra $C$ we mean a co-augmented co-unital (of co-unit $\epsilon$) dg-coalgebra such that $C_n=0$ if $n<0$ or $n=1$ and $C_0=\Bbbk$.  As a vector space, the cobar construction of a dg-coalgebra $(C,\epsilon, d_C,\nabla_C)$ is defined as the free tensor algebra $T(\ds C^+)$. Here, $\ds$ is the usual desuspension, $(\ds M)_n=M_{n+1}$, and $C^+$ is the reduced dg-coalgebra $C^+=\Ker(\epsilon)$.
The differential $d_{\Omega}$  of the cobar construction $\Omega C$ is given as the unique derivation such that its restriction to $\ds C^+$ is 
\begin{align*}
 d_{\Omega}(\ds c)= -\ds d_{C^+}(c)+ (\ds\ot \ds)\nabla_{C^+} (c)~~~ \forall c\in C^+.
\end{align*}

The cobar construction of a dg-coalgebra $C$ can be enriched with a Hopf dg-algebra structure for example when $C$ is endowed with a \emph{homotopy G-\textbf{co}algebra} structure, cf. \cite{Kadeishvili-Mesuring}. 
In this case the resulting double cobar construction $\Omega^2 C=\Omega (\Omega C)$ has a structure of homotopy G-algebra, cf. \cite{Kadeishvili}.
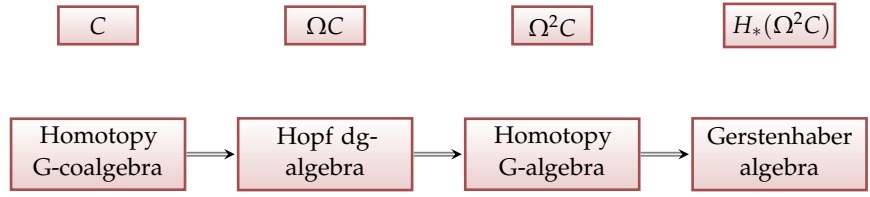
\begin{center}
 \scalebox{0.83}{\begin{tikzpicture} [>=stealth,thick,draw=black!50, arrow/.style={->,shorten >=1pt}, point/.style={coordinate}, pointille/.style={draw=red, top color=white, bottom color=red},
nonterminal/.style={rectangle,minimum size=6mm,very thick,draw=red!50!black!70,top color=white,bottom color=red!60!black!20, }]
% [>=stealth,thick,draw=black!50, arrow/.style={->,shorten >=1pt}, point/.style={coordinate}, pointille/.style={draw=red, top color=white, bottom color=red}]
\matrix[row sep=11mm,column sep=8mm,ampersand replacement=\&]
{ 
 \node [nonterminal,text width=1cm,text centered]  (b0) {$C$}; \& \node  [nonterminal,text width=1cm,text centered]  (c0){$\Omega C$};\& \node[nonterminal,text width=1cm,text centered]  (d0){$\Omega^2 C$};\& \node [nonterminal,text width=1.5cm,text centered] (e0){$H_*(\Omega^2 C)$};\\
 \node [nonterminal,text width=2.5cm,text centered] (b1) {Homotopy G-coalgebra}; \& \node[nonterminal,text width=2.5cm,text centered]  (c1){Hopf dg-algebra};\& \node [nonterminal,text width=2.5cm,text centered] (d1){Homotopy G-algebra};\& \node [nonterminal,text width=2.5cm,text centered] (e1){Gerstenhaber algebra};\\
 }; 
\path
  	  (b1)    edge[above,double,arrow,->]     node {}  (c1)
 	  (c1)    edge[above,double,arrow,->]     node {}  (d1)    
 	 (d1)    edge[above,double,arrow,->]     node {}  (e1)   
        ; 
\end{tikzpicture}
}
\end{center}

A homotopy G-coalgebra structure on a dg-coalgebra $C$ is the data of a coproduct $\nabla: \Omega C \to \Omega C \ot \Omega C$ such that:
\begin{itemize}
 \item $(\Omega C,\nabla)$ is a (co-unitary) dg-bialgebra;
\item  $\nabla$ satisfies the  left-sided condition \eqref{eq: left-side condition} defined in the appendix.
\end{itemize}
The coproduct $\nabla$ being dg-algebra morphism it corresponds to a unique twisting cochain $E: C^+ \to \Omega C \ot \Omega C$.
We note 
\[
 E^{i,j}: C^+\to (C^+)^{\ot i} \ot  (C^+)^{\ot j}
\]
its $(i,j)$-component.
The left-sided condition \eqref{eq: left-side condition} says exactly that $E^{i,j}=0$ when $i\geq 2$.
Both the coassociativity of $\nabla$ and the compatibility with the differential $d_{\Omega}$ lead to relations among the $E^{1,k}$'s. For its particular importance, we distinguish the co-operation $E^{1,1}$ and we denote it by $\nabla_1$.
In particular:
\begin{itemize}
 \item $\nabla_1$ is a chain homotopy for the cocommutativity of the coproduct of $C$;
\item the coassociativity of $\nabla_1$ is controlled by the co-operation $E^{1,2}$. 
\end{itemize}
We refer to the appendix for the complete relationship among these co-operations.
\\

Dually, a homotopy G-algebra structure on a dg-algebra $A$ is the data of a product $\mu: \Ba A \ot \Ba A \to \Ba A$ on the bar construction (see \cite{Kadeishvili}) 
such that:
\begin{itemize}
 \item $(\Ba A,\mu)$ is a (unital) dg-bialgebra;
\item  $\mu$ satisfies a right-sided condition, see \cite[3.1 and 3.2]{LodayRonco}.
\end{itemize}
The product $\mu$ being dg-coalgebra morphism it corresponds to a unique twisting cochain $E: \Ba A \ot \Ba A \to A^+$.
We note 
\[
 E_{i,j}:  (A^+)^{\ot i} \ot  (A^+)^{\ot j}\to A^+
\]
its $(i,j)$-component.
The right-sided condition \cite[3.1]{LodayRonco} says exactly that $E_{i,j}=0$ when $i\geq 2$.
The associativity of $\mu$ and the compatibility with the differential $d_{\Ba}$ of the bar construction leads to relations among the $E_{k,1}$'s.
We distinguish the operation $E_{1,1}$ and we denote it by $\cup_1$.
In particular:
\begin{itemize}
 \item $\cup_1$ is a chain homotopy for the commutativity of the product of $A$;
\item the associativity of $\cup_1$ is controlled by the operation $E_{1,2}$;
\item $\cup_1$ is a left derivation for the product but it is a right derivation up to homotopy for the product.
\end{itemize}
We refer to \cite{Quesney} for explicit relations among the operations, or to \cite{Kadeishvili} in characteristic two.
However, for later use, let us write precisely the three previous facts in characteristic two:
 \begin{align}
  a b -b a&=(\ov{\partial} \cup_1) (a\ot b),\label{diff11}\\
 (a\cup_1 b)\cup_1 c - a\cup_1 (b\cup_1 c)& = E_{1,2}(a;b,c)+E_{1,2}(a;c,b),\label{assoE}\\
  a b\cup_1 c-a(b \cup_1 c)-(a\cup_1 c)b & = 0\label{der11},\\
  a\cup_1 bc -(a\cup_1 b)c - b(a\cup_1 c) & = \ov{\partial} E_{1,2}(a\ot b\ot c),\label{der21}
\end{align}
for all $a,b,c \in \La^+$ and where, for $n\geq 1$, $\ov{\partial}(f)=d_{\La}f-fd_{\La^{\ot n}}$ is the usual differential of  $\Hom(\La^{\ot n},\La)$.
\\
  
 A. Connes and H. Moscovici \cite[(2.21)]{ConnesMosc} defined a boundary operator on the cobar construction of an involutive Hopf algebra (that is a Hopf algebra with an involutive antipode). 
This operator, named 
 the  Connes-Moscovici operator, induces a BV-algebra structure on the homology of such a cobar construction, \cite{MenichiBV}.
In \cite{Quesney} the author showed that the cobar construction of an involutive Hopf dg-algebra $\Hopf$ is endowed with the \emph{Connes-Moscovici structure}. The latter lifts the BV-algebra structure obtained by L. Menichi \cite{MenichiBV} on $H_*(\Omega \Hopf)$.
Moreover, the homotopy G-coalgebras such that the antipode on the resulting cobar construction is involutive was characterized as the \emph{involutive} homotopy G-coalgebras.
%\end{theo}
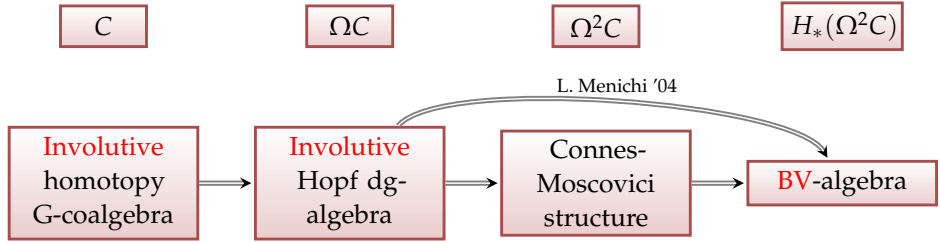
\begin{center}
 \scalebox{0.9}{\begin{tikzpicture} [>=stealth,thick,draw=black!50, arrow/.style={->,shorten >=1pt}, point/.style={coordinate}, pointille/.style={draw=red, top color=white, bottom color=red},
nonterminal/.style={
rectangle,
minimum size=6mm,
very thick,
draw=red!50!black!70,
top color=white,
% a shading that is white at the top...
bottom color=red!60!black!20,}]
\matrix[row sep=11mm,column sep=8mm,ampersand replacement=\&]
{ 
 \node [nonterminal,text width=1cm,text centered]  (b0) {$C$}; \& \node  [nonterminal,text width=1cm,text centered]  (c0){$\Omega C$};\& \node[nonterminal,text width=1cm,text centered]  (d0){$\Omega^2 C$};\& \node [nonterminal,text width=1.5cm,text centered] (e0){$H_*(\Omega^2 C)$};\\
 \node [nonterminal,text width=2.5cm,text centered] (b1) {\textcolor{red}{Involutive} homotopy G-coalgebra }; \& \node[nonterminal,text width=2.5cm,text centered]  (c1){ \textcolor{red}{Involutive} Hopf dg-algebra};\& \node [nonterminal,text width=2.5cm,text centered] (d1){ Connes-Moscovici structure};\& \node [nonterminal,text width=2.5cm,text centered] (e1){ \textcolor{red}{BV}-algebra};\\
 }; 
\path
         % First colonne   
  	  (b1)    edge[above,double,arrow,->]     node {}  (c1)
  	  (c1)    edge[above,double,arrow,->]     node {}  (d1)    
  	  (d1)    edge[above,double,arrow,->]     node {}  (e1)   
	  (c1)    edge[double,arrow,out=50, in=120, looseness=0.4,->]   node[label=above:{\scriptsize L. Menichi '04}] {} (e1)
        ; 
\end{tikzpicture}
}
\end{center} 
\begin{de}
Let $\Hopf$ be an involutive Hopf dg-algebra $\Hopf$. 
 The Connes-Moscovici structure on the cobar construction $\Omega \Hopf$ is the homotopy G-algebra structure  $(\Omega \Hopf, \cup_1, \{E_{1,k}\}_{k\geq 2})$ given in \cite{Kadeishvili} together with the Connes-Moscovici operator $\Delta_{CM}:\Omega \Hopf \to \Omega \Hopf$.
\end{de}
It satisfies the following properties: 
\begin{align}
 &\Delta_{CM}^2=0;\label{eq: CM1}\\
 &(-1)^{|a|}\bigl( \Delta_{CM}(ab)- \Delta_{CM}(a)b - (-1)^{|a|} a\Delta_{CM}(b)\bigr) = \{ a;b\} + \ov{\partial} H(a;b),\label{eq: CM2}
\end{align}
for all homogeneous elements $a,b \in \Omega \Hopf$, where $H$ is a chain homotopy of degree $2$ and 
$\{ a;b\}:= a\cup_1 b -(-1)^{(|a|+1)(|b|+1)} b\cup_1 a$.
\\

The example of homotopy G-coalgebras $C$ that we consider is the simplicial chain complex of a simplicial set.
\\ 
Given a pointed simplicial set $X$, we denote by $C_*(X)$ the simplicial chain complex on $X$. 
The Alexander-Whitney coproduct makes $C_*(X)$ into a dg-coalgebra.
We denote by $|X|$ a geometric realization of $X$. 

H.-J. Baues \cite{Bauesdouble,Bauesgeom} constructed a coproduct 
\[
  \nabla_0:\Omega C_*(X)\to \Omega C_*(X) \ot \Omega C_*(X)
 \]
 on the cobar construction $\Omega C_*(X)$ when the $1$-skeleton of $|X|$ is a point.
It satisfies:
  \begin{itemize}
  \item  $(\Omega C_*(X),\nabla_0)$ is a Hopf dg-algebra;
 \item $H_*(\Omega^2 C_*(X))\cong H_*(\Omega^2 |X|)$ when the $2$-skeleton of $|X|$ is a point.
 \end{itemize}
 This coproduct corresponds to a homotopy G-coalgebra structure on $C_*(X)$.
 \begin{de}
  A homotopy G-coalgebra $(C,\nabla_1,\{E^{1,k}\}_{k\geq 2})$ is reduced if $\nabla_1$ and $E^{1,k}$ for all ${k\geq 2}$ are zero.
 \end{de}
\begin{prop}{\cite[Theorem 4.6]{Quesney}}\label{prop: reduced G-cog}
 Let $\Sigma^2 X$ be a double simplicial suspension. Then
 \begin{itemize}
 \item  Baues' coproduct determines a reduced (hence involutive) homotopy G-coalgebra structure $C_*(\Sigma^2 X)$;
  \item the double cobar construction $\Omega^2 C_*(\Sigma^2 X)$ has the Connes-Moscovici structure.
 \end{itemize} 
 \end{prop}
 The Alexander-Whitney coproduct on $C_*(\Sigma^2 X)$ is primitive, that is all the elements of $C^+_*(\Sigma^2 X)$ are primitive for the coproduct. In fact this is already the case for a suspension $\Sigma X$.
 
\section{Formality theorem}\label{sec: formality th}
In order to prove a formality theorem for the double cobar construction endowed with the Connes-Moscovici structure we need to specify the category of what we call the homotopy BV-algebras. 

\begin{de}
 A homotopy BV-algebra is a homotopy G-algebra $(A,\cup_1,\{E_{1,k}\}_{k\geq 2})$ with a degree $1$ operator $\Delta$ satisfying:
\begin{align}
 &\Delta^2=0;\label{eq: hBV1}\\
 &(-1)^{|a|}\bigl( \Delta(ab)- \Delta(a)b - (-1)^{|a|} a\Delta(b)\bigr) = \{ a;b\} + \ov{\partial} H(a;b),\label{eq: hBV2}
\end{align}
for all homogeneous elements $a,b \in A$, where $H$ is a chain homotopy of degree $2$ and 
$\{ a;b\}:= a\cup_1 b -(-1)^{(|a|+1)(|b|+1)} b\cup_1 a$.
\\
An $\infty$-morphism between two homotopy BV-algebras $A$ and $A'$ is defined as follows.
It is:
\begin{itemize}
 \item a morphism of dg-bialgebras $f:\Ba A\to \Ba A'$ between the associated bar constructions (i.e. an $\infty$-morphism of homotopy G-algebras), its components are $f_n: A^{\ot n} \to A'$, $n\geq 1$; and,
 \item the linear map $f_1: A \to A'$ commutes with the BV-operators, i.e. $f_1\Delta = \Delta f_1$.
\end{itemize}
\end{de}
Here we prove that the double cobar construction of a reduced homotopy G-coalgebra is formal as homotopy BV-algebra in the following sense:
\begin{theo}\label{th: formality}
 Let $\Bbbk$ be any coefficient field. 
  Let $C$ be a reduced (hence involutive) homotopy G-coalgebra with a coproduct such that all the elements of $C^+$ are primitive.
 Then $\Omega^2 C$ and $\Omega^2 H_*(C)$, endowed with the Connes-Moscovici structure, are quasi-isomorphic as homotopy BV-algebras.
 \end{theo}
This theorem relies on the homotopy transfer theorem (HTT for short) for $A_{\infty}$-coalgebras.

The latter transfers the $A_{\infty}$-coalgebra structure along a contraction:
\begin{center}
\begin{tikzpicture} [>=stealth,thick,draw=black!50, arrow/.style={->,shorten >=1pt}, point/.style={coordinate}, pointille/.style={draw=red, top color=white, bottom color=red}]
\matrix[row sep=9mm,column sep=16mm,ampersand replacement=\&]
{
\node (a){$(C,d)$}; \&  \node (b) {$(H,d)$};\\
}; 
\path
         % First colonne 
%(a)     edge[arrow,out=30, in=140, looseness=0.8]      node{\scriptsize$f$} (b)        
(a)     edge[above,arrow,out=30,in=140,looseness=0.8]      node{\scriptsize$p$} (b)
        (b)     edge[above,arrow,out=210, in=320, looseness=0.8]      node{\scriptsize$i$} (a)
	(a)  	   edge[left,arrow,loop above,out=160,in=200,looseness=4]      node[yshift=0.1cm,xshift=-0.1cm]{\scriptsize$\nu$}  (a)
           	; 
\end{tikzpicture}
\end{center}
where $p$ and $i$ are chain maps and $\nu$ is a chain homotopy such that:
\begin{align*}
 pi=Id\\
ip-Id = d\nu + \nu d \\
p\nu=\nu i=\nu^2=0 &~~ \text{(Gauge conditions).}
\end{align*} 
\begin{theo}[Transfer Theorem]\cite{Gugenheim}
 Let $(C,d,\nabla)$ be a connected dg-coalgebra such that $C_1=0$, and $(H,d)$ a dg-vector space. We suppose them related  by a contraction as above.
Then there exists on $H$ an $A_{\infty}$-coalgebra structure $(H,\partial_i)$, and also an $A_{\infty}$-quasi-isomorphism
\[
 f:(C,d,\nabla)\to (H,\partial_i).
\]
The $\partial_i$ are given by:
\begin{align*}
 \partial_0&= d_H \\
\partial_1&=  (p\ot p)\nabla i \\
\partial_2&=p^{\ot 3}(\nabla \nu\ot Id - Id \ot \nabla \nu)\nabla  i,
\end{align*}
and more generally, for $i\geq 2$,
\begin{multline*}
\partial_i= \sum_{\substack{0\leq k_{i-1}\leq i-1 \\ 0\leq k_{i-2}\leq i-2 \\ \cdots \\ 0\leq k_{1}\leq 1 }}\pm  p^{\ot i+1}(Id^{\ot k_{i-1}}\ot \nabla \nu\ot Id^{\ot i-1-k_{i-1}})(Id^{\ot k}\ot \nabla \nu\ot Id^{\ot i-1-k})...\\ ...(\nabla \nu\ot Id - Id \ot \nabla \nu)\nabla i.
\end{multline*}
The quasi-isomorphism $f$ is equivalent to a twisting cochain $\tau: C \to \Omega H$ with components $\tau=\sum_{i\geq 0}\tau_i$ where $\tau_i:C\to H^{\ot i+1}$ is:
\begin{align*}
 \tau_0&= p \\
\tau_1&=  (p\ot p)\nabla \nu \\
\tau_2&=p^{\ot 3}(\nabla\nu\ot Id - Id \ot \nabla\nu)\nabla \nu
\end{align*}
and more generally, for $i\geq 2$,
\begin{multline*}
\tau_i= \sum_{\substack{0\leq k_{i-1}\leq i-1 \\ 0\leq k_{i-2}\leq i-2 \\ \cdots \\ 0\leq k_{1}\leq 1 }}\pm p^{\ot i+1}(Id^{\ot k_{i-1}}\ot \nabla\nu\ot Id^{\ot i-1-k_{i-1}})(Id^{\ot k}\ot \nabla\nu\ot Id^{\ot i-1-k})...\\ ...(\nabla\nu\ot Id - Id \ot \nabla\nu)\nabla \nu.
\end{multline*}
\end{theo}
\begin{proof}[Proof of Theorem \ref{th: formality}]
 The chain complex $C$ is formal as dg-coalgebra.
More precisely, over a field we have a  contraction (cf. \cite{Gugenheim}),
\begin{center}
\begin{tikzpicture} [>=stealth,thick,draw=black!50, arrow/.style={->,shorten >=1pt}, point/.style={coordinate}, pointille/.style={draw=red, top color=white, bottom color=red}]
\matrix[row sep=9mm,column sep=16mm,ampersand replacement=\&]
{
\node (a){$(C,d)$}; \&  \node (b) {$(H_*(C,d),0)$};\\
}; 
\path
         % First colonne 
%(a)     edge[arrow,out=30, in=140, looseness=0.8]      node{\scriptsize$f$} (b)        
(a)     edge[above,arrow,out=30,in=140,looseness=0.8]      node{\scriptsize$p$} (b)
        (b)     edge[above,arrow,out=210, in=320, looseness=0.8]      node{\scriptsize$i$} (a)
	(a)  	   edge[left,arrow,loop above,out=160,in=200,looseness=4]      node[yshift=0.1cm,xshift=-0.1cm]{\scriptsize$\nu$}  (a)
           	; 
\end{tikzpicture}
\end{center}
where, $p$ and $i$ are chain maps and $\nu$ is a chain homotopy such that:
\begin{align}
 pi=Id\\
ip-Id = d\nu + \nu d \\
p\nu=\nu i=\nu^2=0 &~~ \text{(Gauge conditions).}
\end{align}
The transfer theorem for $A_{\infty}$-coalgebra \cite{Gugenheim} transfers the dg-coalgebra structure (seen as an  $A_{\infty}$-coalgebra) from $C$ to $H_*(C,d)$. 
Since the coproduct of $C$ is primitive, the obtained $A_{\infty}$-coalgebra structure  $(H_*(C,d),\delta_i)$ is reduced to the only coproduct 
\begin{equation*}
H(\nabla_{C}):H_*(C,d)\to H_*(C,d)\ot H_*(C,d).
\end{equation*}
Moreover, $H(\nabla_{C})$ is primitive.  
Indeed, the HTT provides a differential (which is also a derivation)
$\delta : \Omega H_*(C,d) \to\Omega H_*(C,d)$. 
This one is given by 
  $\delta=\sum_{i\geq 1}\delta_i$ with $\delta_i(H_*(C,d))\subset (H_*(C,d))^{\ot i+1}$. 
Let us denote $H_*(C,d)$ by $H$.
Explicitly, we have, 
\begin{align*}
 \delta_0&= d_H=0 \\
\delta_1&=  (p\ot p)\nabla_{C} i = \nabla_H \\
\delta_2&=p^{\ot 3}(\nabla_{C} \nu\ot Id - Id \ot \nabla_{C} \nu)\nabla_{C}  i,
\end{align*}
and more generally, for $i\geq 2$,
\begin{multline*}
\delta_i= \sum_{\substack{0\leq k_{i-1}\leq i-1 \\ 0\leq k_{i-2}\leq i-2 \\ \cdots \\ 0\leq k_{1}\leq 1 }}\pm  p^{\ot i+1}(Id^{\ot k_{i-1}}\ot \nabla_{C} \nu\ot Id^{\ot i-1-k_{i-1}})(Id^{\ot k}\ot \nabla_{C} \nu\ot Id^{\ot i-1-k})...\\ ...(\nabla_{C} \nu\ot Id - Id \ot \nabla_{C} \nu)\nabla_{C}  i.
\end{multline*}
In our case, $\nabla_C$ is primitive, and since $\nu$ satisfies the gauge conditions, then only $\delta_1$ is not zero.
Moreover, let us write the twisting cochain $\tau:C\to \Omega H$ provided by the HTT. Its components $\tau=\sum_{i\geq 0}\tau_i$ where $\tau_i:C\to H^{\ot i+1}$ are:
\begin{align*}
 \tau_0&= p \\
\tau_1&=  (p\ot p)\nabla_C \nu \\
\tau_2&=p^{\ot 3}(\nabla_C\nu\ot Id - Id \ot \nabla_C\nu)\nabla_C \nu
\end{align*}
and more generally, for $i\geq 2$,
\begin{multline*}
\tau_i= \sum_{\substack{0\leq k_{i-1}\leq i-1 \\ 0\leq k_{i-2}\leq i-2 \\ \cdots \\ 0\leq k_{1}\leq 1 }}\pm p^{\ot i+1}(Id^{\ot k_{i-1}}\ot \nabla_C\nu\ot Id^{\ot i-1-k_{i-1}})(Id^{\ot k}\ot \nabla_C\nu\ot Id^{\ot i-1-k})...\\ ...(\nabla_C\nu\ot Id - Id \ot \nabla_C\nu)\nabla_C \nu 
\end{multline*}
Thus it is reduced to
\begin{equation*}
\tau_0=p:C\to H. 
\end{equation*}
Thus, $p$ is an $A_{\infty}$-coalgebra morphism, i.e. $\Omega p$ is dg-algebra morphism. Its components $(\Omega p)_n$ on $(\ds C)^{\ot n}$ are given by:
\begin{align*}
  (\Omega p)_n = \ds p\s \ot \ds p\s \ot \cdots \ot \ds p\s .
\end{align*}

We obtain the contraction 
\begin{center}
\begin{tikzpicture} [>=stealth,thick,draw=black!50, arrow/.style={->,shorten >=1pt}, point/.style={coordinate}, pointille/.style={draw=red, top color=white, bottom color=red}]
\matrix[row sep=9mm,column sep=16mm,ampersand replacement=\&]
{
\node (a){$(\Omega C,d)$}; \&  \node (b) {$(\Omega H,0)$};\\
}; 
\path
         % First colonne 
%(a)     edge[arrow,out=30, in=140, looseness=0.8]      node{\scriptsize$f$} (b)        
(a)     edge[above,arrow,out=30,in=140,looseness=0.8]      node{\scriptsize$\Omega p$} (b)
        (b)     edge[above,arrow,out=210, in=320, looseness=0.8]      node{\scriptsize$\Omega i$} (a)
	(a)  	   edge[left,arrow,loop above,out=160,in=200,looseness=4]      node[yshift=0.1cm,xshift=-0.1cm]{\scriptsize$\Gamma$}  (a)
           	; 
\end{tikzpicture}
\end{center}
where, 
\begin{align}
\Omega p \Omega i=Id\\
\Omega i \Omega p-Id = d\Gamma + \Gamma d \\
\Omega p \Gamma =\Gamma \Omega i=\Gamma^2=0 &~~ \text{(Gauge conditions),}
\end{align}
and where  $\Gamma$ is defined by its restrictions $\Gamma_n$ to $(\ds C)^{\ot n}$:
\begin{multline*}
 (\Gamma)_n = \sum_{\substack{k_1+\cdots + k_{s+1}=n-s\\ 1\leq s\leq n}} \pm Id^{\ot k_1} \ot (\ds \nu\s) \ot Id^{\ot k_2} \ot (\ds \ov{\partial}\nu\s) \ot Id^{\ot k_3} \ot (\ds \ov{\partial}\nu\s) \ot Id^{\ot k_4} \\
\ot \cdots \ot  (\ds \ov{\partial} \nu\s) \ot Id^{\ot k_{s+1}},
\end{multline*}
where $\ov{\partial}\nu:=d\nu+\nu d$.
Equivalently, setting $\Gamma_1:=\ds\nu\s$, we have: 
\begin{align*}
 \Gamma_n=Id\ot \Gamma_{n-1} \pm \ds\nu\s\ot \ov{\partial}\Gamma_{n-1} \pm \ds \nu\s\ot Id^{\ot n-1},~n\geq 2.
\end{align*} 
The homotopy G-coalgebra structure of  $C$ endowed $\Omega C$ with a dg-coalgebra structure; we set $(\Omega C,d,\nabla_0)$ the obtained dg-coalgebra. 
This dg-coalgebra structure on $\Omega C$, seen as $A_{\infty}$-coalgebra structure, transfers to an $A_{\infty}$-coalgebra structure on $\Omega H$, which is  $A_{\infty}$-equivalent to the  dg-coalgebra $(\Omega C,d,\nabla_0)$. 
Let us show that the obtained $A_{\infty}$-coalgebra structure  $(\Omega H_*(\Sigma^2 X),\partial_i)$ is reduced to the coproduct $H_*(\nabla_0)$.
The higher coproducts 
\[
\partial_i :  \Omega H \to (\Omega H)^{\ot i+1},
\]
are given by:
\begin{align*}
 \partial_0&= 0 \\
\partial_1&=  (\Omega p\ot \Omega p)\nabla_0 \Omega i =H_*(\nabla_0) \\
\partial_2&=(\Omega p)^{\ot 3}(\nabla_0 \Gamma\ot Id - Id \ot \nabla_0\Gamma )\nabla_0 \Omega i
\end{align*}
and more generally, 
for $i\geq 2$,
\begin{multline*}
\partial_i= \sum_{\substack{0\leq k_{i-1}\leq i-1 \\ 0\leq k_{i-2}\leq i-2 \\ \cdots \\ 0\leq k_{1}\leq 1 }}\pm (\Omega p)^{\ot i+1}(Id^{\ot k_{i-1}}\ot \nabla_0\Gamma\ot Id^{\ot i-1-k_{i-1}})(Id^{\ot k}\ot \nabla_0\Gamma\ot Id^{\ot i-1-k})...\\ ...(\nabla_0\Gamma \ot Id - Id \ot \nabla_0\Gamma)\nabla_0 \Omega i .
\end{multline*}
Let us remark that: each term of $\Gamma$ has at least one operation $\nu$; the coproduct $\nabla_0$, being a shuffle-coproduct, is the composition of  deconcatenations followed by permutations.
Thus, each term of $\nabla_0\Gamma$ has the operation $\nu$.
Then, because of the gauge condition $p\nu=0$, the post-composition by $\Omega p$ vanishes each term of $\nabla_0\Gamma$.\\
Therefore only the co-operation $H_*(\nabla_0)$ is non trivial.

The same argument shows that the $A_{\infty}$-quasi-isomorphism between the dg-coalgebras 
$(\Omega C,d,\nabla_0)$ and $(\Omega H,0,H_*(\nabla_0))$ is strict.
Indeed, the twisting cochain 
\[
\tau:\Omega C\to \Omega^2 H, 
\]
with components $\tau_i:\Omega C\to (\Omega H)^{\ot i+1}$ is:
\begin{align*}
 \tau_0&= \Omega p \\
\tau_1&=  (\Omega p\ot\Omega p)\nabla_0 \Gamma \\
\tau_2&=(\Omega p)^{\ot 3}(\nabla_0\Gamma\ot Id - Id \ot \nabla_0\Gamma)\nabla_0 \Gamma,
\end{align*}
and more generally, for $i\geq 2$,
\begin{multline*}
\tau_i= \sum_{\substack{0\leq k_{i-1}\leq i-1 \\ 0\leq k_{i-2}\leq i-2 \\ \cdots \\ 0\leq k_{1}\leq 1 }}\pm(\Omega p)^{\ot i+1}(Id^{\ot k_{i-1}}\ot \nabla_0\Gamma\ot Id^{\ot i-1-k_{i-1}})(Id^{\ot k}\ot \nabla_0\Gamma\ot Id^{\ot i-1-k})...\\ ...(\nabla_0\Gamma\ot Id - Id \ot \nabla_0\Gamma)\nabla_0 \Gamma. 
\end{multline*}
As a consequence, the double cobar construction $\Omega^2 C$ is quasi-isomorphic to $\Omega^2 H$, as homotopy BV-algebras.
\end{proof}

\section{The 2-restricted G-algebras}\label{sec: G-alg}
In this section the field of coefficients is $\mathbb{F}_2$.
\begin{de}\label{de: lie restreinte}
 Let $r\in \mathbb{N}$. A $2$-restricted Lie algebra of degree $r$ is a connected graded vector space $L$ endowed with a Lie bracket of degree $r$,
\[
 [-;-]_r : L_j\ot L_k \to L_{j+k+r},~~j,k\geq 0,
\]
and a restriction,
\[
 \xi_r:L_n\to L_{2n+r},
\]
satisfying:
\begin{enumerate}
 \item $[-;-]_r$ is symmetric;
 \item $\xi_r(kx)=k^2\xi_r(x)$;
\item $[x;[y;z]_1]_r+[y;[z;x]_r]_r+[z;[x;y]_r]_r=0$;
 \item $[\xi(x);y]_r=[x;[x;y]_r]_r$\label{propt: 4} ; and,
\item $\xi_r(x+y)=\xi_r(x)+[x;y]_1+\xi_r(y)$,\label{propt: 5} 
\end{enumerate}
for all homogeneous elements $x,y,z \in L$.
\end{de}
\begin{de}\label{de: Gerst restreinte}
A $2$-restricted Gerstenhaber algebra  is a degree one $2$-restricted Lie algebra, $G$, and a graded symmetric algebra 
 such that:
\begin{enumerate}
\item $[x,y z]_1 = [x,y]_1z + y[x,z]_1$;
 \item $\xi_1(xy)=x^2\xi_1(y)+ \xi_1(x)y^2 +x[x;y]_1y$,\label{propt: 2}
\end{enumerate} for all homogeneous elements $x,y,z \in G$.
\end{de}

\begin{prop}\label{prop: homologie de G alge est Gerst}
 The homology $H_*(\La,d)$ of a homotopy  G-algebra $(\La,d,\cdot,\cup_1,\{E_{1,k}\}_{k\geq 2})$ over $\mathbb{F}_2$ is a $2$-restricted Gerstenhaber algebra. The restriction is induced by the map $x\mapsto x\cup_1 x$.
\end{prop}
\begin{rmq}
 The restriction $\xi_1$ is a part of Dyer-Lashof operations $\xi_1,\zeta_1$ \cite{FredCohen} in prime characteristic.
 In \cite{Tourtchine} V. Turchin gives explicit formulas for the operations in homotopy G-algebras inducing the Dyer-Lashof operations in homology. He showed via a different method that, in prime characteristic $p$, a homotopy G-algebra $\La$ is, in particular, a $p$-restricted Lie algebra of degree 1.  
\end{rmq}

\begin{proof}
The bracket $[x;y]_1=x\cup_1 y+y\cup_1 x$ together with the multiplication of $\La$ define a Gerstenhaber algebra structure on $H_*(\La)$, \cite{Kadeishvili}.
We set, 
\[
\xi_1(x):=x\cup_1 x, 
\]
for all $x\in \La$. 
To alleviate notation, let us remove temporarily the index $1$ and write $\xi_1$ as $\xi$.  
It suffices to check that $\xi$ induces a homological map satisfying the properties of a restriction.
First, we show the properties \ref{propt: 4} and \ref{propt: 5} from Definition \ref{de: lie restreinte} are satisfied on $\La$, next,
we show that the property \ref{propt: 2} from Definition \ref{de: Gerst restreinte} is satisfied up to homotopy.

We show the equality $[\xi(x);y]_1=[x;[x;y]_1]_1$. 
We have
\begin{equation}\label{eq: restr jacobi}
 \begin{split}
  [x;[x;y]_1]_1&=x\cup_1(x\cup_1 y +y\cup_1 x)+(x\cup_{1}y + y\cup_{1}x) \cup_{1} x\\
&=x\cup_{1} (x\cup_{1} y) + x\cup_{1} (y\cup_{1}x)+ (x\cup_{1} y)\cup_{1}x + (y\cup_{1} x) \cup_{1}x.
 \end{split}
\end{equation}
From the equation \eqref{assoE}, the two first terms satisfy:
\begin{align*}
x\cup_{1}(x\cup_{1}y) &= (x\cup_{1}x)\cup_{1}y + E_{1,2}(x;x,y)+E_{1,2}(x;y,x)~;\\
x\cup_{1} (y\cup_{1}x) &= (x\cup_{1}y) \cup_{1}x  +E_{1,2}(x;y,x)+E_{1,2}(x;x,y).
\end{align*}
Then 
\begin{align*}
x\cup_{1}(x\cup_{1}y)+ x\cup_{1} (y\cup_{1}x) = (x\cup_{1} x)\cup_{1}y + (x\cup_{1}y)\cup_{1}x.
\end{align*}
Similarly, the last term from \eqref{eq: restr jacobi} reads:
\[
(y\cup_{1} x)\cup_{1}x =y\cup_{1}(x\cup_{1}x) +E_{1,2}(y;x,x)+E_{1,2}(y;x,x)=y\cup_{1} (x\cup_{1}x).
\]
Thus,
\begin{align*}
   [x;[x;y]_1]_1=(x\cup_{1} x)\cup_{1}y + y\cup_{1} (x\cup_{1}x) =\xi(x)\cup_{1} y+ y \cup_{1}\xi(x)= [\xi(x);y]_1.
\end{align*}
The following  equality  $\xi(x+y)=\xi(x)+[x;y]_1+\xi(y)$ is straightforward:
\begin{equation}
\begin{split}\label{eq: xi lineaire}
 \xi(x+y)&=(x+y)\cup_{1}(x+y)\\
&= x\cup_{1} x+ y\cup_{1}x +  x\cup_{1}y + y\cup_{1}y\\
&= \xi(x)+[x;y]_1+\xi(y). 
\end{split}
\end{equation}
The equality $\xi(xy)=x^2\xi(y) + \xi(x)y^2 + x[x;y]_1y$ is satisfied up to homotopy on the chain complex $\La$.
This is a consequence of the equations \eqref{der11} and \eqref{der21}. Denoting by $\sim$ the  homotopy relation,
we can write,
\begin{align*}
 \xi(xy)&= xy\cup_{1}xy\\
&= x(y\cup_{1} xy)+ (x\cup_{1} xy)y \\
&\sim x^2(y\cup_{1}y)+x(y\cup_{1}x)y+x(x\cup_{1}y)y +(x\cup_{1}x)y^2\\
&\sim x^2\xi(y)+x[x;y]_1y+\xi(x)y^2 
\end{align*}

Now we show that $\xi$ induces a restriction on the homology $H_*(\La)$.
\\
From the equation \eqref{diff11} the operation $\cup_{1}$ is a homotopy for the commutativity of the product.  
In particular,
\[
 d(x\cup_{1}x)+ dx\cup_{1}x+ x\cup_{1}dx= xx+xx=0, ~~\text{for all}~~ x\in \La.
\]
Thus, if $x\in A$ is a cycle, 
\[
 d\xi(x)=d(x\cup_{1} x)=dx\cup_{1}x+ x \cup_{1} dx=0+0=0.
\]
If $y=dx\in A$ is a boundary, then 
\[
 \xi(y)=y\cup_{1}y=dx\cup_{1} dx=d(dx\cup_{1} x)+d^2x\cup_{1}x + dx\cdot x + xdx = d(dx\cup_{1} x)+d(x^2),
\]
is a boundary.
Thus, for a cycle $x\in \La $ and an element $y \in \La$, we obtain from \eqref{eq: xi lineaire}:
\begin{align*}
 \xi(x+dy)&= \xi(x)+ [x;dy]_1+ \xi(dy) \\
 &= \xi(x)+ d[x;y]_1+ [dx;y]_1+ d(dy\cup_{1} y)\\
&= \xi(x)+ d[x;y]_1+ d(dy\cup_{1}y)+d(y²).
\end{align*}
Therefore $\xi$ induces a restriction in homology.
\end{proof}

\begin{rmq}
 The above proposition also results from the following fact.
 Recall that for a homotopy G-algebra $\La$ the bar construction $\Ba\La$ is a dg-bialgebra. 
 Then, by composition with the restriction morphism $\ca{L}_p\to \ca{AS}$, the bar construction $\Ba \La$ is a $p$-restricted Lie algebra (of degree $0$). 
The projections on $\La$ of both the bracket and the restriction of $\Ba \La$ correspond respectively to the bracket and the  restriction  described in \cite{Tourtchine}.
The latter operations make $\La$ into a $p$-restricted Lie algebra of degree $1$.
In other words, the Lie bracket $[-;-]_0$ and the restriction $\xi_0$ on $\Ba \La$ give, after projection, the degree one Lie bracket $[-;-]_1$ and the restriction $\xi_1$ on $\La$, respectively.
\end{rmq}
\begin{de}
A morphism of $2$-restricted Gerstenhaber algebras  is a degree $0$ linear map  compatible with the  product, the bracket and the restriction.
\end{de}

\begin{de}
 An $\infty$-morphism of homotopy G-algebras $\La$ and $\La'$ is a  morphism of unital dg-bialgebras $f:\Ba\La \to \Ba \La'$. 
\end{de}
Such a morphism is a collection of maps $f_n: \La^{\ot n} \to \La'$, $n\geq1$ satisfying some relations, see \cite[Appendix Definition 4.17]{Quesney} for details. 
In particular, 
\begin{align}
\ov{\partial}f_1&=0;\\
 f_1(x)\cup_1 f_1(y) + f_1(x\cup_1 y)&= f_2(x;y)+ f_2(y;x) ~~\text{for all}~x,y \in \La. \label{eq: morph restr Gerst}
\end{align}

An $\infty$-morphism of homotopy G-algebras induces a morphism of Gerstenhaber algebras in homology, see \cite[Appendix, Proposition 4.20]{Quesney}.
A direct check shows the compatibility of the $\infty$-morphisms with the restriction:
\begin{prop}
 An $\infty$-morphism of homotopy G-algebras induces a morphism of $2$-restricted Gerstenhaber algebras in homology.
\end{prop}
\begin{proof} 
Take $y=x$ in the equation \eqref{eq: morph restr Gerst}, this gives:
\[
 f_1(\xi_1(x))=\xi_1(f_1(x)).
\]
\end{proof}

Now we define the free $2$-restricted  Gerstenhaber algebra. 
\\

Let $L_{1r}$ be a $2$-restricted Lie algebra of degree $1$.
Then the free graded commutative algebra on $L_{1r}$, $S(L_{1r})$ is a $2$-restricted Gerstenhaber algebra.
Indeed, the  Poisson identity
\[
 [xy; z]_1 = x [y;z]_1 + [x;z]_1 y,
\]
is equivalent to the fact that  
\begin{align*}
 ad_1(x):=[-;x]_1: S(L_{1r})&\to S(L_{1r}) \\
       y&\mapsto [y;x]_1 
\end{align*}
is a derivation of degree $+1$ for the algebra $(S(L_{1r}),\cdot)$.
Let us consider a basis of $L_{1r}$ as $k$-vector space. 
It is sufficient to extend the application $ad_1(l):L_{1r}\to L_{1r}\subset S(L_{1r})$ as a derivation on $S(L_{1r})$ for all $l\in L_{1r}$, and next to extend $[x;-]_1:L_{1r}\to S(L_{1r})$ for all $x\in S(L_{1r})$.
\\
The restriction   is extended  via the property \ref{propt: 2} from Definition \ref{de: Gerst restreinte}.
\\

Let $V$ be a dg-vector space with $V_0=0$. If $L_{1r}(V)$ denotes the free degree one $2$-restricted Lie algebra on $V$, then $S(L_{1r}(V))$ is the free $2$-restricted Gerstenhaber algebra on %\note{to be proved} 
$V$.

\section{The homology of the double cobar construction of a double suspension over \texorpdfstring{$\mathbb{F}_2$}{Lg}}
\begin{theo}\label{th: GerstLibreCobar1restreint}
Let the coefficient field be $\mathbb{F}_2$. 
Then $\Omega^2 C_*(\Sigma^2X)$ has a Connes-Moscovici structure. 
The induced  $2$-restricted Gerstenhaber algebra structure on $H_*(\Omega^2 C_*(\Sigma^2 X))$ is free on the reduced homology  $H^+_*(X)$.
\end{theo}
It was proved in \cite{FredCohen} that $H_*(\Omega^2\Sigma^2 Y)$ is the free $2$-restricted Gerstenhaber algebra on $H_*^+(Y)$, then we obtain
\begin{cor}
 The restricted Gerstenhaber algebras $H_*(\Omega^2 C_*(\Sigma^2X))$ and  $H_*(\Omega^2|\Sigma^2 X|)$ are isomorphic. 
\end{cor}

\begin{proof}[Proof of Theorem \ref{th: GerstLibreCobar1restreint}]
 From Proposition \ref{prop: reduced G-cog} and the formality Theorem  \ref{th: formality} the  homotopy G-algebras $\Omega^2 C_*(\Sigma^2 X)$ and $\Omega^2 H_*(\Sigma^2 X)$ are quasi-isomorphic. Therefore, it remains to prove that the structure of homotopy  G-algebra on the double cobar construction  $\Omega^2 H_*(\Sigma^2 X)$ induces the free $2$-restricted Gerstenhaber algebra on $H^+_*(X)$.

We denote by $[-;-]_1$ the Gerstenhaber bracket on $H_*(\Omega^2 H_*(\Sigma^2 X))$,
\[
 [x;y]_1=x\cup_1 y+ y\cup_1 x,
\]
and we denote by $\xi_1$ the restriction,
\[
 \xi_1(x)=x\cup_1 x,
\]
for all $x\in H_*(\Omega^2 H_*(\Sigma^2 X))$.

We have an inclusion $\iota : H^+_*(X)\cong \s^{-2} {H}^+_*(\Sigma^2 X)\hookrightarrow H_*(\Omega^2 H_*(\Sigma^2X))$.
Thus, from the freeness of the functors $S$ and $L_{1r}$, we have the commutative diagram:
\begin{center}
%\begin{equation}\label{diag: objet libre}
\begin{tikzpicture} [>=stealth,thick,draw=black!50, arrow/.style={->,shorten >=1pt}, point/.style={coordinate}, pointille/.style={draw=red, top color=white, bottom color=red}]
\matrix[row sep=9mm,column sep=16mm,ampersand replacement=\&]
{
  \node (a2){$SL_{1r}(H^+_*(X))$}; \& \node (b1) {$H_*(\Omega^2 H_*(\Sigma^2X))$} ;\\
    \node (b2){$H^+_*(X)$} ; \\
}; 
\path
         % First colonne 
	  
   	 (a2)     edge[above,arrow,->]      node {$SL_{1r}(\iota)$}  (b1)
	 (b2)     edge[right,arrow,->]      node {$i$}  (a2)
 	 (b2)     edge[below,arrow,->]      node {$\iota$}  (b1)
        ; 
\end{tikzpicture}
%\end{equation}
\end{center}
where $SL_{1r}(\iota)$ is a  morphism of $2$-restricted  Gerstenhaber  algebras.  
It remains to show that $SL_{1r}(\iota)$ is an isomorphism.
To do this, we construct two spectral sequences $\ov{E}$ and $E$, and a morphism $\Phi:E\to \ov{E}$ of spectral sequences  satisfying:
\begin{center}
\begin{tikzpicture} [>=stealth,thick,draw=black!50, arrow/.style={->,shorten >=1pt}, point/.style={coordinate}, pointille/.style={draw=red, top color=white, bottom color=red}]
\matrix[row sep=9mm,column sep=26mm,ampersand replacement=\&]
{
   \node (b1) {$T H^+_*(\Sigma X)\ot SL_{1r}H^+_*(X)=E^2$} ; \&   \node (b2){$\ov{E}^2=H_*(\Omega {H}_*(\Sigma^2 X))\ot H_*(\Omega^2H_*(\Sigma^2 X)).$} ; \\
}; 
\path
         % First colonne 
	  
   	  (b1)     edge[above,arrow,->]      node {$T(\iota)\ot S(L_{1r}(\iota))$}  (b2)
 	 %(b2)     edge[below,arrow,->]      node {$$}  (b1)
        ; 
\end{tikzpicture}

\end{center}
We conclude by mean of the comparison theorem \cite[Theorem 11.1 p.355]{MacLaneBookHomology}, using that both $T(\iota)$ and $\Phi^{\infty}$ are isomorphisms.\\

Let us consider the application
\begin{align*}
 \pi :\Omega H_*(\Sigma^2 X)\ot\Omega^2H_*(\Sigma^2 X) \to \Omega H_*(\Sigma^2 X)
\end{align*}
defined by
$\pi(x\ot y)=\epsilon(y)x$, where $\epsilon$ is the augmentation of $\Omega H_*(\Sigma^2 X)$.
The total space 
\[
\mathcal{E}:=\Omega H_*(\Sigma^2 X)\ot\Omega^2H_*(\Sigma^2 X) 
\]
is endowed with a differential $d$ and a contracting homotopy $s$ defined as in \cite[Axioms \eqref{R1},\eqref{R2},\eqref{D1},\eqref{D2}]{AdamsHilton}, see below for the details. 
The total space is acyclic \cite[Lemme 2.3]{AdamsHilton}.
We set
\[
 B=\Omega H_*(\Sigma^2 X)
\]
and 
\[
 A=\Omega^2 H_*(\Sigma^2 X)
\]
so that $\mathcal{E}=B\ot A$.
We define a filtration of $\mathcal{E}$ as
\[
F_r\mathcal{E}=  \bigoplus_{p\leq r}B_p\ot A.
\]
Following \cite{AdamsHilton}, we get a spectral sequence $\ov{E}^r$ associated to this filtration. The abutment is the homology of $\mathcal{E}$, the second page  is 
\begin{align*}
\ov{E}^2_{p,q}=H_p(\Omega {H}_*(\Sigma^2 X))\ot H_q(\Omega^2H_*(\Sigma^2 X)).
\end{align*}
The  total space being acyclic we have $\ov{E}^{\infty}_{0,0}=\mathbb{F}_2$, $\ov{E}^{\infty}_{p,q}=0$ if $(p,q)\neq (0,0)$.
\\
The following  spectral sequence of dg-algebras is used by {F. Cohen} \cite[III,p.228]{FredCohen} as a model for the Serre  spectral sequence associated to the paths fibration, 
\[
\Omega^{n+1}\Sigma^{n+1} X\to P\Omega^{n} \Sigma^{n+1} X \to \Omega^n \Sigma^{n+1} X,~~ n>0. 
\]
He shows that the homology $H_*(\Omega^{n+1}\Sigma^{n+1} X)$ is the free $2$-restricted Gerstenhaber algebra of degree $n$  on $H_*^+(X)$.
\\

We defined the spectral sequence  $E^r$. Its second page is given by
\begin{align*}
E^2_{p,q}= (T({H}^+_*(\Sigma X)))_p\ot (S(L_{1r}(H^+_*(X))))_q,
\end{align*}
 the   differentials are defined such that  $E^r$ is a spectral sequence of dg-algebras whose transgressions are
\begin{align*}
t(\Sigma x)&= x \\ 
t(\xi_0^k(\Sigma x))&= \xi^k_1(x)\\
t(ad_0(\Sigma x_1)\cdots ad_0(\Sigma x_{k-1})(\Sigma x_{k}))&=ad_1(x_1)\cdots ad_1( x_{k-1})( x_{k}),% 
\end{align*}
for all  $\Sigma x,\Sigma x_1,...,\Sigma x_k \in {H}^+_*(\Sigma X)$,
where:
\begin{itemize}
  \item $\xi_0(x)=x^2$ is the restriction of $TH^+_*(\Sigma X)$;%, the free $2$-restricted Lie algebra of degree $0$;
 \item $\xi_1$ is the restriction of  $SL_{1r}(H^+_*(X))$; 
\item $[-;-]_0$ is the Lie bracket of  $TH^+_*(\Sigma X)$;
 \item $[-;-]_1$ is the Gerstenhaber bracket of $SL_{1r}H^+_*(X)$; and,
 \item $ad_j(x)(y)=[y;x]_j$ for $j=0$ or $1$.\\
\end{itemize}
Then the $r$-th transgression $t=d^r:E^r_{r,0}\to E^r_{0,r-1}$ is extended to 
$ d^r:E^r_{p,q}\to E^r_{p-r,q+r-1}$ as a derivation.

Let us investigate the example of one generator $\Sigma x\in H^+_n(\Sigma X)$ of degree $n$; the spectral sequence is: 
%\begin{center}
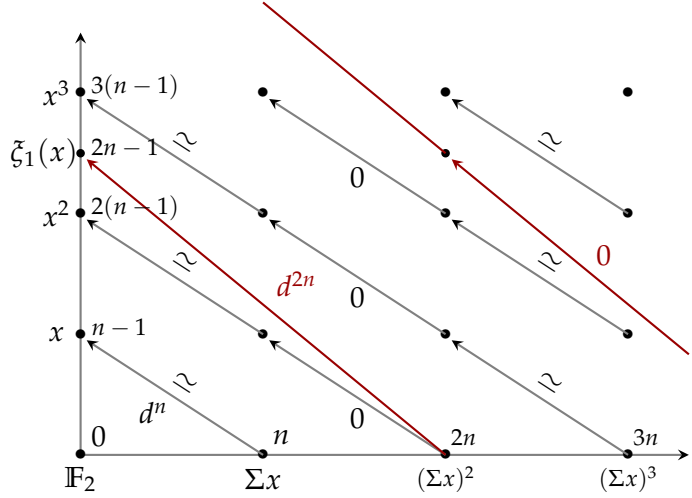
\begin{figure}[H]\label{fig: spect gene 1}\centering
\begin{tikzpicture}[>=stealth,thick,draw=black!50, arrow/.style={->,shorten >=1pt}, point/.style={coordinate}, pointille/.style={draw=red, top color=white, bottom color=red},scale=0.8]
%\draw[help lines] (0,0) grid (9,6);
%  \pgfpathcircle{\pgfpoint{3cm}{0cm}} {2pt}
% \pgfpathcircle{\pgfpoint{6cm}{0cm}} {2pt}
% \pgfpathcircle{\pgfpoint{9cm}{0cm}}{2pt}
% \pgfusepath{fill}
% \pgfpathmoveto{\pgfpointxy{0}{0}}
% \pgfpathlineto{\pgfpointxy{9}{0}}
% \pgfusepath{stroke}

\draw[->] (0,0) -- (0,7);
\draw[->] (0,0) -- (10,0);

\pgfpathcircle{\pgfpoint{0cm}{5cm}}{2pt}
\pgfpathcircle{\pgfpoint{6cm}{5cm}}{2pt}
 \pgfusepath{fill}
%\pgftext[\pgfpoint{0.2cm}{5.2cm}]{5}

\foreach \t in {2,3}
{\pgftext[at=
\pgfpointlineattime{\t}{\pgfpoint{0.3cm}{+0.3cm}}{\pgfpoint{3.3cm}{+0.3cm}}]{$\t n$}
}
\foreach \t in {0,3,...,9}
{
\pgftext[at=
\pgfpointlineattime{\t}{\pgfpoint{0cm}{0cm}}{\pgfpoint{1cm}{0cm}}]{$\bullet$}
} 
\foreach \d in {2,4,...,6}
{%\pgftext[at=
%\pgfpointlineattime{\d}{\pgfpoint{0.15cm}{0.2cm}}{\pgfpoint{0.15cm}{1.2cm}}]{\d}
\foreach \t in {0,3,...,9}
{%\pgftext[at=
%\pgfpointlineattime{\t}{\pgfpoint{0.2cm}{+0.2cm}}{\pgfpoint{1.2cm}{+0.2cm}}]{\t}
\pgftext[at=
\pgfpointlineattime{\t}{\pgfpoint{0cm}{\d cm}}{\pgfpoint{1cm}{\d cm}}]{$\bullet$}
}
%\pgftext[at=
%\pgfpointlineattime{\d}{\pgfpoint{0cm}{0cm}}{\pgfpoint{0cm}{1cm}}]{$\bullet$}
}
\foreach \d in {0,2,...,4}
{
\draw[->] (6,\d) -- node[below] {$0$} (3+0.1,\d+1.9);
\foreach \t in {0,6}
{\draw[->] (3+\t,\d) -- node[above,sloped] {$\simeq$} (\t+0.1,\d+1.9);}}
\foreach \t in {2,3}{
\pgftext[at=
\pgfpointlineattime{\t}{\pgfpoint{0cm}{-0.4cm}}{\pgfpoint{3cm}{-0.4cm}}]{$(\Sigma x)^{\t}$}}
%\draw[->] (6,0) -- (3.1,1.9);
\draw[->,red!60!black] (6,0) -- node[above right]{$d^{2n}$} (0.1,4.9);
\draw[-,red!60!black] (6,5) --  (3,7.5);
\draw[->,red!60!black] (10,10/6) --  (6.1,4.9);
\draw (0,5.1) node[right]{\begin{footnotesize}$2n-1$\end{footnotesize}} ;
\draw (0,-0.4) node[]{$\mathbb{F}_2$} ;
\draw (3,-0.4) node[]{$\Sigma x$} ;

\draw (1.2,0.7) node[]{$d^n$} ;

\draw (0.3,0.3) node[]{$0$} ;
\draw[red!60!black] (8.6,3.3) node[]{$0$} ;
\draw (3.3,0.3) node[]{$n$} ;

\draw (-0.4,2) node[]{$x$} ;
\draw (-0.4,4) node[]{$x^2$} ;
\draw (-0.4,6) node[]{$x^3$} ;
\draw (-0.6,5) node[]{$\xi_1(x)$} ;

\draw (+0,2.1) node[right]{\begin{footnotesize}$n-1$\end{footnotesize}} ;
\draw (+0,4.1) node[right]{\begin{footnotesize}$2(n-1)$\end{footnotesize}} ;
\draw (+0,6.1) node[right]{\begin{footnotesize}$3(n-1)$\end{footnotesize}} ;
\end{tikzpicture}
\end{figure}
All the brackets are zero since $[\Sigma x;\Sigma x]_0=2(\Sigma x)^2=0$ and $[\xi_0(\Sigma x); \Sigma x]_0= [\Sigma x;[\Sigma x;\Sigma x]_0]_0$~; it remains: the element $\Sigma x$, the restriction $\xi_0(\Sigma x)=(\Sigma x)^2$ and its   iterations, plus the resulting products.
This spectral sequence split  
into elementary spectral sequences on the transgressive elements $\Sigma x$, $\xi_0(\Sigma x)$, $\xi_0^2(\Sigma x)$,... with  differentials the corresponding transgression.
Indeed, $T(\Sigma x)$ can be written additively as the exterior algebra on the transgressive elements $\Sigma x,\xi_0^k(\Sigma x)$, $k\geq 1$. Also, $S(L_{1r}(x))$ can be written additively as polynomial algebra on the image of  these transgressive elements $x,\xi_1^k(x)$, $k\geq 1$. 
Then the  elementary  spectral sequences are of the form
\[
 \Lambda(y)\ot P(t(y)),
\]
where $y$ run through the transgressive elements of $T(\Sigma x)$. 
Such a presentation of  $T(\Sigma x)$ can be extended  to  $TH^+_*(\Sigma X)$ via the $[-;-]_0$-\emph{elementary products}.
This presentation  is due to P. J. Hilton \cite{Hilton}. 
Let us consider the Lie algebra  $L(H^+_*(\Sigma X))$ without restriction, and with bracket $[-;-]_0$.
The  $[-;-]_0$-{elementary products} are defined as follows. \\
The elements  $x\in H^+_*(\Sigma X)$ are the $[-;-]_0$-elementary product of weight  $1$.
Suppose that the $[-;-]_0$-{elementary products} of weight $j$ have been defined for $j<k$.
Then, a product of weight $k$ is an element $[x;y]_0$ such that:
\begin{enumerate}
\item $x$ is a $[-;-]_0$-{elementary products} of weight $u$;
\item $y$ is a $[-;-]_0$-{elementary products} of weight $v$;
\item $u+v=k$;
%\item $[x;y]_0$ est de poids $k$~; et,
\item $x<y$, and if $y=[z;t]_0$ for $z$ and $t$ elementary, then $z\leq x$.
\end{enumerate}
We order the generators $(\Sigma x_i)_i$ of $H^+_*(\Sigma X)$ by their dimension; the $[-;-]_0$-{elementary products}  are the elements 
\[
\Sigma x_k, \cdots,  [\Sigma x_k;\Sigma x_l]_0, \cdots, [\Sigma x_j;[\Sigma x_k;\Sigma x_l]_0]_0,\cdots 
\]
where $j<k<l$. 
From \cite{Hilton} we obtain 
\[
L(H^+_*(\Sigma X))\cong \text{span}\{ y~|~ y~\text{a}~[-;-]_0-\text{elementary product}\}, 
\]
as vector spaces.
Denoting  by $p_j$ the  $[-;-]_0$-elementary products we have the isomorphism of vector spaces
\[
 TH^+_*(\Sigma X)\cong \bigotimes_j P(p_j),
\]
where $P$ denotes the polynomial algebra.
To take into account the restriction, we write $P(p_j)$ as $\wedge_{k\geq 0} (p_j^{2^k})$ where $\wedge$ denotes  the exterior algebra\footnote{To avoid any ambiguity, we  recall that over  $\mathbb{F}_2$, the exterior algebra $\wedge V$  of a vector space $V$ is defined as the  tensor algebra $TV$ quotiented by the ideal generated by $\{ v\ot v ~|~ v\in TV\}$.}.
We obtain a decomposition into $[-;-]_0$-elementary products and $\xi_0$-\emph{elementary products}, 
\[
 TH^+_*(\Sigma X)\cong \bigotimes_j \bigwedge_{k\geq 0} (p_j^{2^k}).
\]
On the other hand, we recall that $L_{1r}=\ds L_{r}\s$, that is $[x;y]_1=\ds[\s x;\s y]_0$ and $\xi_1(x)=\ds \xi_0(\s x)$.
Then we obtain that 
\[
 SL_{1r}(H^+_*(\Sigma X))\cong S(\text{span}\{ y,\xi_1^k(y)~|~ y~\text{is a}~[-;-]_1-\text{elementary product},~k\geq1\}).
\]
Then the  spectral sequence $E$ decomposes into tensor products of elementary spectral sequences 
\[
 \wedge(y)\ot P(t(y)),
\]
where $y$ run through the set of  $[-;-]_0$-elementary products and  $\xi_0$-elementary products.

Thus the spectral sequence $E$ converges to $E^{\infty}_{0,0}=\mathbb{F}_2$, $E^{\infty}_{p,q}=0$ if $(p,q)\neq (0,0)$.
\\

Return to the application 
\[
 T(\iota)\ot S(L_{1r}(\iota)): E^{2}\to \ov{E}^{2}.
\]
From the comparison theorem \cite[Theorem 11.1 p. 355]{MacLaneBookHomology},
it is sufficient to show that: the application $\Phi^2:=T(\iota)\ot S(L_{1r}(\iota))$ induces morphism of spectral sequences $\Phi:E\to \ov{E}$; the morphism $T(\iota)$ is an isomorphism; and $\Phi^{\infty}:E^{\infty}\to \ov{E}^{\infty}$ is an isomorphism.
\\

We show that the application $\Phi^2$:
\begin{center}
\begin{tikzpicture} [>=stealth,thick,draw=black!50, arrow/.style={->,shorten >=1pt}, point/.style={coordinate}, pointille/.style={draw=red, top color=white, bottom color=red}]
\matrix[row sep=9mm,column sep=26mm,ampersand replacement=\&]
{
   \node (b1) {$T H^+_*(\Sigma X)\ot SL_{1r}H^+_*(X)$} ; \&   \node (b2){$H_*(\Omega {H}_*(\Sigma^2 X))\ot H_*(\Omega^2H_*(\Sigma^2 X))$} ; \\
}; 
\path
         % First colonne 
	  
   	  (b1)     edge[above,arrow,->]      node {$T(\iota)\ot S(L_{1r}(\iota))$}  (b2)
 	 %(b2)     edge[below,arrow,->]      node {$$}  (b1)
        ; 
\end{tikzpicture}
\end{center}
induces a morphism of spectral sequences.

To do this, let us write with details the differential of $\mathcal{E}=B\ot A$.
The axioms (R1) and (R2) determine a contracting homotopy on $\mathcal{E}$:
\begin{align}
 s(1)&= 0,~~ s(\ds x)= x,~~ s(x)=0,~~& & \text{for}~~ \ds x\in A\tag{R1}\label{R1}\\
 s(xy)&= s(x)y + \epsilon(x)s(y),~~& &\text{for}~~ x\in \mathcal{E}, y\in A.\tag{R2}\label{R2}
\end{align}
The axioms (D1) and (D2)  determine he differential of $\mathcal{E}$:
\begin{align}
 d(x)&=(1-sd)(\ds x),~~ &  &\text{for}~~ x\in B \tag{D1}\label{D1}\\
 d(xy)&= d(x)y +(-1)^{|x|} xd(y),~~ & &\text{for}~~ x\in \mathcal{E}, y\in A.\tag{D2}\label{D2}
\end{align}
Then, for all ${x\in \ds H^+_{p+1}(\Sigma^2 X) \subset B_p}$, we obtain 
\[
d(x)= \ds x\in \s^{-2}H^+_{p+1}(\Sigma^2X)\subset F_0\mathcal{E}, 
\]
since the coproduct $\nabla_0$ is primitive on $H^+_{*}(\Sigma^2 X)$.
Thus, for such elements $x$ of degree $p$, we have $d^r(x)=0$ if $r< p$.
Moreover, for all ${x\in \ds H^+_{p+1}(\Sigma^2 X)\subset \ov{E}^p_{p,0} \subset B_p}$, we obtain 
\[
d^p(x)= \ds x\in \s^{-2}H^+_{p+1}(\Sigma^2X){\subset \ov{E}^p_{0,p-1} \subset A_{p-1}}. 
\]
More generally, the degree $p$ elements in $B_p=\Omega H_{*}(\Sigma^2 X)_p$, which are (after desuspension) cycles  for $d_A$, belong to $\ov{E}^p_{p,0}$. 
On the other hand, from the construction (see \cite[section 4]{Kadeishvili}) of the operation 
\[
\cup_1:\Omega^2 H_*(\Sigma^2 X)\ot \Omega^2 H_*(\Sigma^2X) \to \Omega^2 H_*(\Sigma^2 X), 
\]
the Gerstenhaber bracket  $[-;-]_1=\cup_1 \circ (1-\tau)$ is  compatible with suspensions, i.e. for all $x,y\in {\s^{-2} H^+_*(\Sigma^2 X)\subset \ds\Omega H_*(\Sigma^2 X)}$, we have $[x;y]_1=\ds [\s x;\s y]_0$. 
And similarly for the restriction $\xi_1(x)=\ds \xi_0(\s x)$.
Then
\begin{align*}
t(ad_0(\ds x_1)\cdots ad_0(\ds x_{k-1})(\ds x_{k}))&=ad_1(x_1)\cdots ad_1(x_{k-1})(x_{k}) \\
t(\xi_0^k(\ds x))&=\xi_1^k(x),
\end{align*}
{for all} $\ds x, \ds x_1,...,\ds x_k \in \ds {H}^+_*(\Sigma^2 X)$.
Finally, since $\Phi^2$ is an algebra morphism, it commutes with the  differentials.
Therefore $\Phi$ is a morphism of spectral sequences.
\\
The morphisms $\Phi^{\infty}$ and  $T(\iota)$ are clearly  isomorphisms. This ends the proof.
\end{proof}
\newpage
\section{The homology of the double cobar construction of a double suspension over  \texorpdfstring{$\mathbb{Q}$}{Lg}}
In this section the coefficient field is $\mathbb{Q}$.
\subsection{Free Gerstenhaber algebra structure}
\begin{theo}\label{th: GerstLibreCobar}
Let the coefficient field be $\mathbb{Q}$. 
Then $\Omega^2 C_*(\Sigma^2X)$ has a Connes-Moscovici structure. 
The induced  Gerstenhaber algebra structure on $H_*(\Omega^2 C_*(\Sigma^2 X))$ is free on the reduced homology  $H^+_*(X)$.
\end{theo}
It was proved in \cite{FredCohen} that $H_*(\Omega^2\Sigma^2 Y)$ is the free Gerstenhaber algebra on $H_*^+(Y)$, then we obtain
\begin{cor}
 The  Gerstenhaber algebras $H_*(\Omega^2 C_*(\Sigma^2X))$ and  $H_*(\Omega^2|\Sigma^2 X|)$ are isomorphic. 
\end{cor}

\begin{proof}[Proof of Theorem \ref{th: GerstLibreCobar}]
 The proof is analogous to that of Theorem \ref{th: GerstLibreCobar1restreint}.
We replace the functor $L_{1r}$ by $L_1$ adjoint of the forgetful functor from the category of degree $1$ Lie algebras to the category  of graded vector spaces.
The symmetric algebra $S(V)$ on a graded vector space $V$, is $TV/I$ where $I$ is the ideal generated by the elements $v\ot w -(-1)^{|v||w|} w\ot v$. Then $S(L_1(V))$ is a  Gerstenhaber algebra; its bracket is  $[v,w]_1=vw -(-1)^{(|v|-1)(|w|-1)}wv$ for $v,w\in V$.
The main difference is the decomposition of the model spectral sequence $E$. 
Indeed, on a generator $\Sigma x\in H^+_n(\Sigma X)$ of  even degree $n$, the  spectral sequence is:
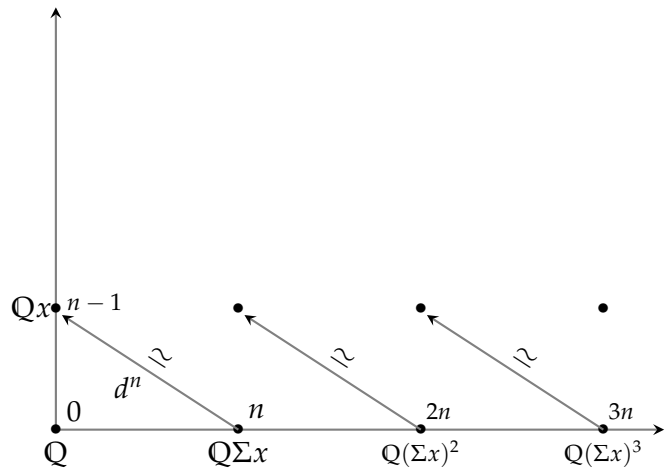
\begin{center}
\scalebox{1}{
\begin{tikzpicture}[>=stealth,thick,draw=black!50, arrow/.style={->,shorten >=1pt}, point/.style={coordinate}, pointille/.style={draw=red, top color=white, bottom color=red},scale=0.8]

\draw[->] (0,0) -- (0,7);
\draw[->] (0,0) -- (10,0);

% \pgfpathcircle{\pgfpoint{0cm}{5cm}}{2pt}
% \pgfpathcircle{\pgfpoint{6cm}{5cm}}{2pt}
%  \pgfusepath{fill}
%\pgftext[\pgfpoint{0.2cm}{5.2cm}]{5}

\foreach \t in {2,3}
{\pgftext[at=
\pgfpointlineattime{\t}{\pgfpoint{0.3cm}{+0.3cm}}{\pgfpoint{3.3cm}{+0.3cm}}]{$\t n$}
}
\foreach \t in {0,3,...,9}
{
\pgftext[at=
\pgfpointlineattime{\t}{\pgfpoint{0cm}{0cm}}{\pgfpoint{1cm}{0cm}}]{$\bullet$}
} 
\foreach \d in {2}
{%\pgftext[at=
%\pgfpointlineattime{\d}{\pgfpoint{0.15cm}{0.2cm}}{\pgfpoint{0.15cm}{1.2cm}}]{\d}
\foreach \t in {0,3,...,9}
{%\pgftext[at=
%\pgfpointlineattime{\t}{\pgfpoint{0.2cm}{+0.2cm}}{\pgfpoint{1.2cm}{+0.2cm}}]{\t}
\pgftext[at=
\pgfpointlineattime{\t}{\pgfpoint{0cm}{\d cm}}{\pgfpoint{1cm}{\d cm}}]{$\bullet$}
}
%\pgftext[at=
%\pgfpointlineattime{\d}{\pgfpoint{0cm}{0cm}}{\pgfpoint{0cm}{1cm}}]{$\bullet$}
}
\foreach \d in {0}
{
\draw[->] (6,\d) -- node[above,sloped] {$\simeq$} (3+0.1,\d+1.9);
\foreach \t in {0,6}
{\draw[->] (3+\t,\d) -- node[above,sloped] {$\simeq$} (\t+0.1,\d+1.9);}}
\foreach \t in {2,3}{
\pgftext[at=
\pgfpointlineattime{\t}{\pgfpoint{0cm}{-0.4cm}}{\pgfpoint{3cm}{-0.4cm}}]{$\mathbb{Q}(\Sigma x)^{\t}$}}
%\draw[->] (6,0) -- (3.1,1.9);
%\draw[->,red] (6,0) -- node[above right]{$d^{2n}$} (0.1,4.9);
\draw (0,-0.4) node[]{$\mathbb{Q}$} ;
\draw (3,-0.4) node[]{$\mathbb{Q}\Sigma x$} ;

\draw (1.2,0.7) node[]{$d^n$} ;

\draw (0.3,0.3) node[]{$0$} ;
\draw (3.3,0.3) node[]{$n$} ;

\draw (-0.4,2) node[]{$\mathbb{Q}x$} ;

\draw (+0,2.1) node[right]{\begin{footnotesize}$n-1$\end{footnotesize}} ;
% \draw (+0,4.1) node[right]{\begin{footnotesize}$2(n-1)$\end{footnotesize}} ;
% \draw (+0,6.1) node[right]{\begin{footnotesize}$3(n-1)$\end{footnotesize}} ;
\end{tikzpicture}}
\end{center}
On a generator $\Sigma x\in H^+_n(\Sigma X)$ of odd degree $n$, the spectral sequence is:
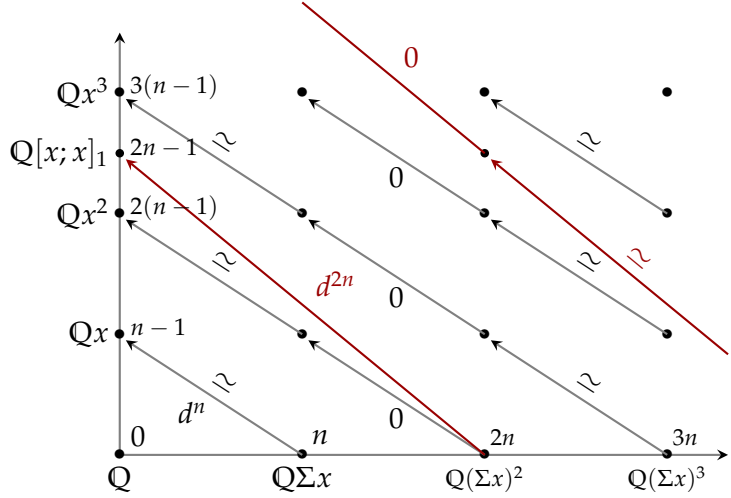
\begin{center}
\scalebox{1}{
\begin{tikzpicture}[>=stealth,thick,draw=black!50, arrow/.style={->,shorten >=1pt}, point/.style={coordinate}, pointille/.style={draw=red, top color=white, bottom color=red},scale=0.8]
%\draw[help lines] (0,0) grid (9,6);
%  \pgfpathcircle{\pgfpoint{3cm}{0cm}} {2pt}
% \pgfpathcircle{\pgfpoint{6cm}{0cm}} {2pt}
% \pgfpathcircle{\pgfpoint{9cm}{0cm}}{2pt}
% \pgfusepath{fill}
% \pgfpathmoveto{\pgfpointxy{0}{0}}
% \pgfpathlineto{\pgfpointxy{9}{0}}
% \pgfusepath{stroke}

\draw[->] (0,0) -- (0,7);
\draw[->] (0,0) -- (10,0);

\pgfpathcircle{\pgfpoint{0cm}{5cm}}{2pt}
\pgfpathcircle{\pgfpoint{6cm}{5cm}}{2pt}
 \pgfusepath{fill}
%\pgftext[\pgfpoint{0.2cm}{5.2cm}]{5}

\foreach \t in {2,3}
{\pgftext[at=
\pgfpointlineattime{\t}{\pgfpoint{0.3cm}{+0.3cm}}{\pgfpoint{3.3cm}{+0.3cm}}]{$\t n$}
}
\foreach \t in {0,3,...,9}
{
\pgftext[at=
\pgfpointlineattime{\t}{\pgfpoint{0cm}{0cm}}{\pgfpoint{1cm}{0cm}}]{$\bullet$}
} 
\foreach \d in {2,4,...,6}
{%\pgftext[at=
%\pgfpointlineattime{\d}{\pgfpoint{0.15cm}{0.2cm}}{\pgfpoint{0.15cm}{1.2cm}}]{\d}
\foreach \t in {0,3,...,9}
{%\pgftext[at=
%\pgfpointlineattime{\t}{\pgfpoint{0.2cm}{+0.2cm}}{\pgfpoint{1.2cm}{+0.2cm}}]{\t}
\pgftext[at=
\pgfpointlineattime{\t}{\pgfpoint{0cm}{\d cm}}{\pgfpoint{1cm}{\d cm}}]{$\bullet$}
}
%\pgftext[at=
%\pgfpointlineattime{\d}{\pgfpoint{0cm}{0cm}}{\pgfpoint{0cm}{1cm}}]{$\bullet$}
}
\foreach \d in {0,2,...,4}
{
\draw[->] (6,\d) -- node[below] {$0$} (3+0.1,\d+1.9);
\foreach \t in {0,6}
{\draw[->] (3+\t,\d) -- node[above,sloped] {$\simeq$} (\t+0.1,\d+1.9);}}
\foreach \t in {2,3}{
\pgftext[at=
\pgfpointlineattime{\t}{\pgfpoint{0cm}{-0.4cm}}{\pgfpoint{3cm}{-0.4cm}}]{$\mathbb{Q}(\Sigma x)^{\t}$}}
%\draw[->] (6,0) -- (3.1,1.9);
\draw[->,red!60!black] (6,0) -- node[above right]{$d^{2n}$} (0.1,4.9);
\draw[-,red!60!black] (6,5) --  node[above right]{$0$} (3,7.5);
\draw[->,red!60!black] (10,10/6) -- node[above right,sloped]{$\simeq$} (6.1,4.9);
\draw (0,5.1) node[right]{\begin{footnotesize}$2n-1$\end{footnotesize}} ;
\draw (0,-0.4) node[]{$\mathbb{Q}$} ;
\draw (3,-0.4) node[]{$\mathbb{Q}\Sigma x$} ;

%\draw[text,red] (8.6,3.3) node[]{$0$} ;

\draw (1.2,0.7) node[]{$d^n$} ;

\draw (0.3,0.3) node[]{$0$} ;
\draw (3.3,0.3) node[]{$n$} ;

\draw (-0,2) node[left]{$\mathbb{Q}x$} ;
\draw (-0,4) node[left]{$\mathbb{Q}x^2$} ;
\draw (-0,6) node[left]{$\mathbb{Q}x^3$} ;
\draw (-0,5) node[left]{$\mathbb{Q}[x;x]_1$} ;

\draw (+0,2.1) node[right]{\begin{footnotesize}$n-1$\end{footnotesize}} ;
\draw (+0,4.1) node[right]{\begin{footnotesize}$2(n-1)$\end{footnotesize}} ;
\draw (+0,6.1) node[right]{\begin{footnotesize}$3(n-1)$\end{footnotesize}} ;
\end{tikzpicture}}
\end{center}
The spectral sequence $E$ decomposes into elementary spectral sequences
\[
 P(y)\ot S(t(y)) ~~\text{and}~~ S(z)\ot P(t(z))
\]
where $y$ and $z$ run through the respectively  even and odd degree  $[-;-]_0$-elementary products of $TH^+_*(\Sigma X)$.
\end{proof}
\subsection{BV-algebra structure}\label{BV-alg}
\subsubsection{Free Batalin-Vilkovisky algebra}\label{sec: free BV}
We fix a connected vector space $V$ and an operator $\Delta: V_* \to V_{*+1}$ such that $\Delta^2=0$.
We define the free Batalin-Vilkovisky algebra on $(V,\Delta)$ as
\[
 S(L_1(V\oplus \Delta(V))).
\]
The operator $\Delta$ is extended as BV-operator via the formula 
\[
 \Delta(xy)=\Delta(x)y+ (-1)^{|x|}x\Delta(y)+(-1)^{|x|}[x;y]_1,
\] 
for all homogeneous elements $x,y \in L_1(V)$, and
\[
 \Delta([x;y]_1)=[\Delta(x);y]_1+(-1)^{|x|-1}[x;\Delta(y)]_1,
\]
for all homogeneous elements  $x,y \in L_1(V)$.
\\
If $\Delta:V\to V$ is trivial then $\Delta:  S(L_1(V)) \to S(L_1(V))$ is the BV-operator given by the extension of the  Gerstenhaber bracket. The resulting  Batalin-Vilkovisky structure is called \emph{canonical Batalin-Vilkovisky  algebra structure} $S(L_1(V))$, cf. \cite[Section 1.1]{TamarkinTsygan}.
\\
Moreover, if the free Gerstenhaber algebra $S(L_1(V))$ is a Batalin-Vilkovisky algebra such that the BV-operator is zero when it is restricted to $V$, then $S(L_1(V))$ is the canonical Batalin-Vilkovisky algebra.

\subsubsection{Batalin-Vilkovisky algebra on the homology of a double loop space}
Let us make explicit the Batalin-Vilkovisky structure  given by E. Getzler \cite{Getzler} on $H_*(\Omega^2\Sigma^2X)$.\\

Let $Y$ be a $S^1$-space, that is a pointed topological space with base point $*$, endowed with an action of the circle  $S^1$ such that $g\cdot *=*$ for all $g\in S^1$.
E. Getzler shows \cite{Getzler} that the  diagonal action of $S^1$ on $\Omega^2 Y=\Hom(S^2,Y)$:
\begin{align*}
\Delta: S^1\times \Hom(S^2,Y)&\to \Hom(S^2,Y)\\
 (g,\phi)&\mapsto (g\cdot \phi)(s):=g\cdot \phi (g^{-1}\cdot s),
\end{align*}
   induces a Batalin-Vilkovisky algebra structure on $H_*(\Omega^2 Y)$.
 For a double suspension $Y=\Sigma^2 X=S^2\wedge X$,
we endowed $Y$ with the  canonical action of  $S^1$ on $S^2$ and the trivial action  on $X$ i.e. $g\cdot (s\wedge x)=(g\cdot s)\wedge x$ for all $g\in S^1$ et $s\wedge x \in S^2\wedge X$.
Then the adjunction $X\to \Omega^2 \Sigma^2 X$ is equivariant for the trivial action on $X$ and the diagonal action on  $\Omega^2 \Sigma^2 X$. 
As a consequence, cf. \cite{MenichiGaudens}, ($H_*(\Omega^2\Sigma^2X),\Delta)$ is a canonical Batalin-Vilkovisky algebra, free on  $H_*^+(X)$.

\subsubsection{Homology of double cobar construction}
\begin{theo}\label{th: BV sur 2Cobar}
 Let the coefficient field be $\mathbb{Q}$.
 Then $\Omega^2 C_*(\Sigma^2X)$ has a Connes-Moscovici structure. The induced BV-algebra structure on $(H_*(\Omega^2 C_*(\Sigma^2 X)),\Delta_{CM})$ is free on the reduced homology $H^+_*(X)$.
\end{theo}
As a consequence,
\begin{cor}
The Batalin-Vilkovisky algebras $(H_*(\Omega^2 C_*(\Sigma^2X)),\Delta_{CM})$ and $(H_*(\Omega^2|\Sigma^2 X|),\Delta)$ are isomorphic.
\end{cor}
\begin{proof}[Proof of  Theorem \ref{th: BV sur 2Cobar}]
From Proposition \ref{prop: reduced G-cog} and  the formality Theorem \ref{th: formality}, $\Omega^2 H_*(\Sigma^2 X)$ has a Connes-Moscovici structure and its homology  $H_*(\Omega^2 H_*(\Sigma^2 X))$ is a Batalin-Vilkovisky algebra isomorphic to $H_*(\Omega^2 C_*(\Sigma^2 X))$.
\\
On the other hand, the  simplicial set $X$ is seen as endowed with a trivial action of $S^1$. 
Then the inclusion
\[
\iota: (H^+_*(X),0)\to H_*(\Omega^2 H_*(\Sigma^2 X)), 
\]
lifts to a unique morphism of Batalin-Vilkovisky algebras,  
\[
 SL(\iota) : S(L_1(H^+_*(X))) \to H_*(\Omega^2 H_*(\Sigma^2 X)),
\]
with the usual commutative diagram.

Let us observe the operator 
\[
 H(\Delta_{CM}): H_*(\Omega^2 H_*(\Sigma^2 X))\to H_*(\Omega^2 H_*(\Sigma^2 X)),
\]
induced by the Connes-Moscovici operator $\Delta_{CM}:\Omega^2 H_*(\Sigma^2 X)\to \Omega^2 H_*(\Sigma^2 X)$.
By construction we have (see \cite[(2.22)]{ConnesMosc}): 
\[
  \Delta_{CM}|_{\ds \Omega H_*(\Sigma^2 X)}=0~~~\text{and \textit{a fortiori}}~~~\Delta_{CM}|_{\s^{-2}H^+_*(\Sigma^2 X)}=0.
\]
Then
\[
 H(\Delta_{CM})|_{H^+_*(X)}=0.
\]
Moreover, from  Theorem \ref{th: GerstLibreCobar}, the underlying  Gerstenhaber algebra  $H_*(\Omega^2 H_*(\Sigma^2 X))$ is the Gerstenhaber algebra, free on $H^+_*(X)$.
This concludes the proof by the argument from section \ref{sec: free BV}.
\end{proof}

\section{Appendix}
In this appendix we write down the whole relations among the co-operations defining a Hirsch coalgebra.
Homotopy G-coalgebras are particular Hirsch coalgebras.

\begin{de}
A Hirsch coalgebra $(C,d,\nabla_{C},E=\{E^{i,j}\})$ is the data of a coproduct $\nabla_{E}:\Omega C\to \Omega C \ot \Omega C$ on the cobar  construction of the $1$-connected dg-coalgebra $(C,d,\nabla_{C})$ such that $\Omega C$ is a co-unital dg-bialgebra. 
\end{de} 
The  coproduct $\nabla_{E}$ corresponds to a twisting cochain 
\[
E:\ds C^+ \to \Omega C\ot \Omega C. 
\]
The $(i,j)$-component is 
\[
E^{i,j}:\ds C^+\to (\ds C^+)^{\ot i}\ot (\ds C^+)^{\ot j}. 
\]
To alleviate notation, we set $\ov{C}:=\ds C^+$.\\
Thus
\[
E^{i,j}:\ov{C}\to \ov{C}^{\ot i}\ot \ov{C}^{\ot j}. 
\]
The degree of $E^{i,j}$ is $0$.

Let us make explicit the co-unit condition, the  compatibility with the differential and the  coassociativity of the  coproduct.
\subsubsection{The co-unit condition}
The following relation $(\epsilon \ot 1)\nabla_E=(1\ot \epsilon )\nabla_E=Id$ gives:
\begin{align}
 E^{0,1}=E^{1,0}=Id_{\ov{C}} ~~~\text{and} ~~~ E^{0,k}=E^{k,0}=0 ~~\text{for}~k>1.
\end{align}

\subsubsection{The compatibility with the differential}
\newcommand{\gau}{\overleftarrow}
 \newcommand{\droi}{\overrightarrow}
We investigate the equation $(d_{\Omega}\ot 1 + 1\ot d_{\Omega})\nabla_E= \nabla_E d_{\Omega}$.
\\
\\
On the component  $\ov{C}^{\ot k}\ot \ov{C}^{\ot l}\subset \Omega C \ot \Omega C$, we obtain:
\begin{equation}\label{eq: Leibniz}
 \begin{split}
  d_{\ov{C}^{\ot k}\ot \ov{C}^{\ot l}}E^{k,l}- E^{k,l}d_{\ov{C}}=
 &-\sum_{i+2+j=k}((1^{\ot i}\ot \nabla_{\ov{C}}\ot 1^{\ot j})\ot 1^{\ot l})E^{k-1,l}\\
 &- \sum_{i+2+j=l}(1^{\ot k}\ot(1^{\ot i}\ot \nabla_{\ov{C}}\ot 1^{\ot j}))E^{k,l-1} \\
&+ \sum_{\substack{k_1+k_2=k \\ l_1+l_2=l}} \mu_{\Omega \ot \Omega }(E^{k_1,l_1}\ot E^{k_2,l_2})\nabla_{\ov{C}},
 \end{split}
\end{equation}
%\end{multline}
where $k_1+l_1>0$ and $k_2+l_2>0$.
\\

The terms $\mu_{\Omega \ot \Omega }(E^{k_1,l_1}\ot E^{k_2,l_2})\nabla_{\ov{C}}$ read as follows.
We set
\[
E^{k,l}(a)= a^{|\gau{k}}\ot a^{|\droi{l}} ~~~\text{and}~~~ \nabla_{\ov{C}}(a)=a^1\ot a^2.
\]
Then
\[
 \mu_{\Omega \ot \Omega }(E^{k_1,l_1}\ot E^{k_2,l_2})\nabla_{\ov{C}}=\pm (a^{1|\gau{k_1}}\ot a^{2|\gau{k_2}}) \ot ( a^{1|\droi{l_1}}\ot a^{2|\droi{l_2}}).
\]

On the component $\ov{C}\ot \ov{C}\subset \Omega C\ot \Omega C $ the equation \eqref{eq: Leibniz} gives:
\begin{equation}\label{Leibniz2}
 \begin{split}
  d_{\ov{C}\ot \ov{C}}E^{1,1}- E^{1,1}d_{\ov{C}}= \nabla_{\ov{C}} + \tau\nabla_{\ov{C}};
 \end{split}
\end{equation}
that is, on $C^+\ot C^+$:
\begin{equation}\label{Leibniz1}
 \begin{split}
  d_{C^+\ot C^+}E_+^{1,1} +  E_+^{1,1}d_{C^+}= \nabla_{C^+} - \tau\nabla_{C^+},
 \end{split}
\end{equation}
where $E^{1,1}=(\ds\ot \ds)E_+^{1,1}\s $.\\
Thus $E_+^{1,1}$ is a chain homotopy for the  cocommutativity of the coproduct $\nabla_{C^+}$.
We distinguish it from the  co-operations $E^{i,j}$ and we denote it by $\nabla_1$ as a dual analogue to the  $\cup_1$-product.

\subsubsection{The coassociativity condition}
On the  component $\ov{C}^{\ot i}\ot \ov{C}^{\ot j} \ot \ov{C}^{\ot k}$ we obtain:
\begin{equation}\label{eq: coass Hirsch general}
\begin{split}
 \sum_{n=0}^{i+j} \sum_{\substack{i_1+...+i_n=i\\ j_1+...+j_n =j}}(\mu^{(n)}_{\Omega\ot \Omega}( E^{i_1,j_1}\ot...\ot E^{i_n,j_n})\ot 1^{\ot k})E^{n,k}\\
=\sum_{n=0}^{j+k} \sum_{\substack{j_1+...+j_n =j \\ k_1+...+k_n=k}}( 1^{\ot i}\ot\mu^{(n)}_{\Omega\ot \Omega}( E^{j_1,k_1}\ot...\ot E^{j_n,k_n}))E^{i,n}.
\end{split}
\end{equation}
In particular, for $i=j=k=1$, we have:
\begin{equation*}
 (E^{1,1}\ot 1)E^{1,1} +(1\ot 1)E^{2,1} +(\tau\ot 1)E^{2,1} =(1\ot E^{1,1})E^{1,1}+(1\ot 1)E^{1,2}+(1\ot \tau)E^{1,2}.
\end{equation*}
That is,
\begin{equation}\label{eq: coass Hirsch}
 (\nabla_1\ot 1)\nabla_1 -(1\ot \nabla_1)\nabla_1 =   (1\ot (1+\tau))E^{1,2}-((1+\tau)\ot 1)E^{2,1}.
\end{equation}
The lack of coassociativity of $\nabla_1$ is controlled by the co-operations $E^{1,2}$ and $E^{2,1}$.

\begin{de}
 A homotopy G-coalgebra $(C,d,\nabla_C,E^{i,1})$ is a Hirsch coalgebra such that $E^{i,j}=0$ for $i\geq 2$.
\end{de}
\begin{prop} 
Let $(C,d,\nabla_C,E^{i,j})$ be a Hirsch coalgebra. Then $E^{i,j}=0$ for $i\geq 2$ is equivalent to the following left-sided condition:
 \begin{equation}\label{eq: left-side condition}
  \forall r \in \mathbb{N}, ~J_r:=\oplus_{n\leq r} \ov{C}^{\ot n} ~~\text{is a left co-ideal for the  coproduct} ~\nabla_{E}: \Omega C\to \Omega C \ot \Omega C.
 \end{equation}
\end{prop}
\begin{proof}
 If $J_r$ is a left co-ideal then $\nabla_E(J_r)\subset J_r\ot \Omega C$. Then for $r=1$, we have $\nabla_E(\ov{C})= \sum_{i,j} E^{i,j}(\ov{C})\subset \ov{C}\ot \Omega C$. This implies that $E^{i,j}=0$ for $i\geq 2$. 
Conversely, if $E^{i,j}=0$ for $i\geq 1$, then 
\begin{multline*}
\nabla_E(c_1\ot \cdots \ot c_n)= \sum_{i_1,\cdots ,i_n,j_1,\cdots , j_n} \mu_{\Omega}^{(n)}(E^{i_1,j_1}\ot \cdots \ot  E^{i_n,j_n})(c_1\ot \cdots c_n) 
\end{multline*}
belongs to $\oplus \ov{C}^{i_1+\cdots + i_n}\ot \Omega C$. Since $i_s\leq 1$ for $1\leq s\leq n$, then $i_1+\cdots + i_n \leq n$.
\end{proof}

 \bibliographystyle{alpha}
%\bibliography{/datas/Latex/biblio/biblio-utf8} %
\bibliography{biblio-utf8}

\begin{thebibliography}{{Que}13}

\bibitem[AH56]{AdamsHilton}
J.~F. Adams and P.~J. Hilton.
\newblock On the chain algebra of a loop space.
\newblock {\em Comment. Math. Helv.}, 30:305--330, 1956.

\bibitem[Bau81]{Bauesdouble}
H.~J. Baues.
\newblock The double bar and cobar constructions.
\newblock {\em Compositio Math.}, 43(3):331--341, 1981.

\bibitem[Bau98]{Bauesgeom}
Hans-Joachim Baues.
\newblock The cobar construction as a {H}opf algebra.
\newblock {\em Invent. Math.}, 132(3):467--489, 1998.

\bibitem[CM00]{ConnesMosc}
Alain Connes and Henri Moscovici.
\newblock Cyclic cohomology and {H}opf algebra symmetry.
\newblock {\em Lett. Math. Phys.}, 52(1):1--28, 2000.
\newblock Conference Mosh{\'e} Flato 1999 (Dijon).

\bibitem[Coh76]{FredCohen}
Frederick Cohen.
\newblock The homology of ${C}_{n+1}$-spaces, $n\geq 0$, in the homology of
  iterated loop spaces.
\newblock Lecture Notes in Math.(Vol.533):207--351, 1976.

\bibitem[Get94]{Getzler}
E.~Getzler.
\newblock Batalin-{V}ilkovisky algebras and two-dimensional topological field
  theories.
\newblock {\em Comm. Math. Phys.}, 159(2):265--285, 1994.

\bibitem[GM08]{MenichiGaudens}
Gerald Gaudens and Luc Menichi.
\newblock Batalin-{V}ilkovisky algebras and the {$J$}-homomorphism.
\newblock {\em Topology Appl.}, 156(2):365--374, 2008.

\bibitem[Gug82]{Gugenheim}
V.~K. A.~M. Gugenheim.
\newblock On a perturbation theory for the homology of the loop-space.
\newblock {\em J. Pure Appl. Algebra}, 25(2):197--205, 1982.

\bibitem[Hil55]{Hilton}
P.~J. Hilton.
\newblock On the homotopy groups of the union of spheres.
\newblock {\em J. London Math. Soc.}, 30:154--172, 1955.

\bibitem[Kad04]{Kadeishvili-Mesuring}
T.~Kadeishvili.
\newblock Measuring the noncommutativity of {DG}-algebras.
\newblock {\em J. Math. Sci. (N. Y.)}, 119(4):494--512, 2004.
\newblock Topology and noncommutative geometry.

\bibitem[Kad05]{Kadeishvili}
T.~Kadeishvili.
\newblock On the cobar construction of a bialgebra.
\newblock {\em Homology Homotopy Appl.}, 7(2):109--122, 2005.

\bibitem[LR10]{LodayRonco}
Jean-Louis Loday and Mar{\'{\i}}a Ronco.
\newblock Combinatorial {H}opf algebras.
\newblock In {\em Quanta of maths}, volume~11 of {\em Clay Math. Proc.}, pages
  347--383. Amer. Math. Soc., Providence, RI, 2010.

\bibitem[Men04]{MenichiBV}
Luc Menichi.
\newblock Batalin-{V}ilkovisky algebras and cyclic cohomology of {H}opf
  algebras.
\newblock {\em $K$-Theory}, 32(3):231--251, 2004.

\bibitem[ML95]{MacLaneBookHomology}
Saunders Mac~Lane.
\newblock {\em Homology}.
\newblock Classics in Mathematics. Springer-Verlag, Berlin, 1995.
\newblock Reprint of the 1975 edition.

\bibitem[{Que}13]{Quesney}
A.~{Quesney}.
\newblock Homotopy {BV}-algebra structure on the double cobar construction.
\newblock {\em ArXiv e-prints (submitted)}, May 2013.

\bibitem[Tou06]{Tourtchine}
Victor Tourtchine.
\newblock Dyer-{L}ashof-{C}ohen operations in {H}ochschild cohomology.
\newblock {\em Algebr. Geom. Topol.}, 6:875--894 (electronic), 2006.

\bibitem[TT00]{TamarkinTsygan}
D.~Tamarkin and B.~Tsygan.
\newblock Noncommutative differential calculus, homotopy {BV} algebras and
  formality conjectures.
\newblock {\em Methods Funct. Anal. Topology}, 6(2):85--100, 2000.

\end{thebibliography}
\end{document}